\documentclass{article}
\usepackage{amsmath,amssymb}
\usepackage{amscd}
\usepackage[a4paper,nohead,headsep=0.6cm,left=2cm,right=2cm,top=2cm,
bottom=2cm,twoside]{geometry}
\usepackage{amsthm}
\usepackage{euscript}
\usepackage{latexsym}
\usepackage[matrix,arrow,curve,cmtip]{xy}
\usepackage[cp1251]{inputenc}
\usepackage[english]{babel}
\usepackage{wrapfig}
\usepackage{mathrsfs}
\usepackage{graphicx}
\usepackage{keyval}
\usepackage{BOONDOX-frak}
\usepackage{mathtools}
\usepackage{tikz}
\usetikzlibrary{arrows,shapes,snakes,automata,backgrounds,petri}
\usepackage{verbatim}
\newtheorem{theorem}{Theorem}[section]
\newtheorem{remark}{Remark}[section]

\newtheorem{definition}{Definition}[section]
\newtheorem{example}{Example}[section]
\newtheorem{proposition}{Proposition}[section]

\begin{document}
\markboth{Viktor Lopatkin}
{Derivations of a Leavitt Path Algebra $W(n)$}

\title{DERIVATIONS OF A LEAVITT PATH ALGEBRA $W(\ell)$}

\author{\footnotesize VIKTOR LOPATKIN\footnote{wickktor@gmail.com}}

\maketitle

\abstract{The aim of this paper is to describe all inner and all outer derivations of Leavitt path algebra via explicit formulas.}

\section*{Introduction}
The Leavitt path algebras were developed by Gene Abrams and Arando Pino \cite{Lev} and Pere Ara, Mar\'ia A. Moreno and Enrique Pardo \cite{Lev2}. These algebras are an algebraic analog of graph Cuntz--Kreiger $C^*$-algebra.

\smallskip

\par In \cite{AraGood} it has been proved that the Leavitt path algebra and their generalizations are hereditary algebra. It follows that their homology vanish in higher degree. Pere Ara and Guillermo Corti\~nas \cite{AraCortinas} calculated the Hochschild homology of Leavitt path algebras. But they used a technique of spectral sequences, and from their results is not possible to get explicit formulas for derivations of these algebra. In this paper we will describe all derivations via an explicit formula. We will give a full description of the space of all inner and all outer derivations.

\smallskip

\par The main technique for the describing all derivations is based on the Gr\"obner--Shirshov  basis and the Composition--Diamond Lemma \cite{BokSurv}. If the Gr\"obner--Shirshov basis for an algebra $\Lambda$ is known, then a basis $\mathfrak{B}_\Lambda$ for this algebra is also known. It allows us to describe any value of a linear map $f:\Lambda \to \Lambda$ as a decomposition $f(\lambda) = \sum\limits_{x \in\mathfrak{B}_\Lambda}\xi_x(\lambda)x$ via basic elements; of course, in the infinite-dimensional case we have to assume that almost all scalars $\xi_x(\lambda)$ are zero. Since the derivation is a linear map which satisfies Leibnitz rule, we can use the Gr\"obner--Shirshov basis to find exactly the needed conditions for the scalars $\xi_x(\lambda)$. This is the main tool in this paper. The Gr\"obner--Shirshov basis of Leavitt path algebra $L(\Gamma)$ was found in \cite{Zel}.

\smallskip

\par The main results of this paper are Theorem \ref{innerder} and Theorem \ref{outerder} which describe all inner and all outer derivations of the Leavitt path algebra $W(\ell)$, respectively.

\section{Preliminaries}
Here we remind the definition of the Leavitt path algebra and the correspondence terminology.

\smallskip

\paragraph{Derivations.} Let $\mathfrak{A}$ be an arbitrary (non-associative) algebra over the ring $R$. A {\it derivation} $\mathscr{D}$ of $\mathfrak{A}$ is a linear map $\mathscr{D}:\mathfrak{A} \to \mathfrak{A}$ satisfying to Leibnitz rule,
\[
\mathscr{D}(xy) = \mathscr{D}(x)y + x \mathscr{D}(y)
\]
for any $x,y\in\mathfrak{A}$. It follows that any derivation $\mathscr{D}$ is uniquely determined by its values on the generators of algebra $\mathfrak{A}$. Moreover, let us assume that the basis $\mathfrak{B}_\mathfrak{A}$ of the algebra $\mathfrak{A}$ is given, say $\mathfrak{B}_\mathfrak{A} = \{b_j,\,j\in J\}$, then for any generator $x$, we can put
\[
\mathscr{D}(x) = \sum\limits_{b_i \in \mathfrak{B}_\mathfrak{A}}\xi_{b_i}(x)b_i,
\]
where $\xi_{b_i}(x) \in R$ are scalars, and almost all of them are zero.

\smallskip

\par Now let $\mathfrak{A}$ be an associative algebra, and let us fix some element $\lambda\in\mathfrak{A}$. The {\it inner derivation determined by} $\lambda$ is a linear map $\mathrm{ad}_\lambda:\mathfrak{A} \to \mathfrak{A}$ which is defined for any element $x \in \mathfrak{A}$ as follows:
\[
\mathrm{ad}_\lambda(x):=\lambda x - x \lambda.
\]

\smallskip

\par This allows to define any inner derivation $A$ for any generator $x$ as follows:
\[
A(x) = \sum\limits_{\lambda \in\mathfrak{A}}\zeta_\lambda \mathrm{ad}_\lambda(x),
\]
where almost all scalars $\zeta_\lambda\in R$ are zero.

\smallskip

\paragraph{Leavitt path algebra $L(\Gamma)$.} A directed graph $\Gamma = (V,E, \mathfrak{s}, \mathfrak{r})$ consists of two sets $V$ and $E$, called vertices and edges respectively, and two maps $\mathfrak{s}, \mathfrak{r}:E \to V$ called {\it source} and {\it range} (of edge) respectively. The graph is called {\it row-finite} if for all vertices $v \in V$, $|\mathfrak{s}^{-1}(v)| < \infty$. A vertex $v$ for which $\mathfrak{s}^{-1}(v)$ is empty is called a {\it sink}. A {\it path} $p = e_1\cdots e_\ell$ in a graph $\Gamma$ is a sequence of the edges $e_1, \ldots, e_\ell \in E$ such that $\mathfrak{r}(e_i) = \mathfrak{s}(e_{i+1})$ for $i = 1, \ldots, \ell -1$. In this case we say that the path $p$ {\it starts} at the vertex $\mathfrak{s}(e_1)$ and {\it ends} at the vertex $\mathfrak{r}(e_\ell)$,  and put $\mathfrak{s}(p): = \mathfrak{s}(e_1)$ and $\mathfrak{r}(p):=\mathfrak{r}(e_\ell)$. We also set $p_0 : = e_1$ and $p_z : = e_\ell$. Further, we will use the following notation: we set $p/p_0: = p'$ and $p/p_z = p''$, where the paths $p',p''$ can be defined as $p_0p' = p$ and $p''p_z = p$, respectively. In the case $p = p_0 \in E$ or $p = p_z \in E$, then we set $p/p_0:=\mathfrak{r}(p_0)$ and $p/p_z: = \mathfrak{s}(p_z)$, respectively.

\smallskip

\begin{definition} Let $\Gamma$ be a row-finite graph, and let $R$ be an associative ring with unit. The Leavitt path $R$-algebra $L_R(\Gamma)$ (or, shortly, $L(\Gamma)$) is the $R$-algebra given by the set of generators $\{v, v \in V\}$, $\{e,e^*| e \in E\}$ and the set of relations:
\begin{itemize}
\item[{\rm 1)}] $v_iv_j = \delta_{i,j}v_i$, for all $v_i,v_j \in V$;
\item[{\rm 2)}] $\mathfrak{s}(e)e = e\mathfrak{r}(e) = e,$ $\mathfrak{r}(e)e^* = e^* \mathfrak{s}(e) = e^*,$ for all $e \in E;$
\item[{\rm 3)}] $e^*f = \delta_{e,f}\mathfrak{r}(e),$ for all $e,f \in E;$
\item[{\rm 4)}] $v = \sum\limits_{\mathfrak{s}(e)=v}ee^*$, for an arbitrary vertex $v \in V \setminus\{\mathrm{sinks}\}.$
\end{itemize}
\end{definition}

\smallskip

\par For an arbitrary vertex $v \in V$ which is not a sink, choose and an edge $\varphi(v)\in E$ such that $\mathfrak{s}(\varphi(v)) = v$. We will refer to this edge as {\it special}. In other words, we fix a function $\varphi:V \setminus \{\mbox{sinks}\} \to E$ such that $\mathfrak{s}(\varphi(v)) = v$ for an arbitrary $v \in V \setminus \{\mbox{sinks}\}$.

\smallskip

\par In \cite{Zel} the Gr\"obner--Shirshov basis of the Leavitt path algebra $L(\Gamma)$ has been obtained with respect to the order $<$ on the set of generators $X = V \cup E \cup E^*$. This order is defined as follows: chose an arbitrary well-ordering on the set of vertices $V$. If $e$, $f$ are edges and $\mathfrak{s}(e) < \mathfrak{s}(f)$, then $e < f$. It remains to order the edges that have the same source. Let $v$ be a vertex which is not a sink. Let $e_1, \ldots, e_\ell$ be all the edges that originate from $v$. Suppose $e_1$ is a special edge. We order the edges as follows: $e_1 > e_2 >\ldots > e_\ell$. Choose an arbitrary well-ordering on the set $E^*$. For arbitrary elements $v\in V$, $e \in E^*$, $f^* \in E^*$, we let $v<e<f^*$. Thus the set $X = V \cup E \cup E^*$ is well-ordered.

\smallskip

\begin{theorem}{\cite[Theorem 1]{Zel}}\label{GSBLev}
The following elements form a basis of the Leavitt path algebra $L(\Gamma)${\rm :}
\begin{itemize}
\item[{\rm (i)}] the set of all vertices $V$,
\item[{\rm (ii)}] the set of all paths $\mathfrak{P}$,
\item[{\rm (iii)}] the set $\mathfrak{P}^* := \{p^*: p \in \mathfrak{P}\},$
\item[{\rm (iv)}] the set $\mathfrak{M}$ of words of the form $wh^*$, where $w = e_1\cdots e_n \in \mathfrak{P}$, $h^* = (f_1 \cdots f_m)^* = f_m^* \cdots f_1^* \in \mathfrak{P}^*$, $e_i,f_j \in E$, are paths that end at the same vertex $\mathfrak{r}(w) = \mathfrak{r}(h)$, with the condition that the edges $e_n$ and $f_m$ are either distinct or equal, but not special.
\end{itemize}
\end{theorem}

\smallskip

\par Let us describe the Gr\"obner--Shirshov basis for the Leavitt path algebra $L(\Gamma)$:
\begin{align*}
& vu = \delta_{v,u}v,\\
& ve = \delta_{v,\mathfrak{s}(e)}e, \quad ev = \delta_{v, \mathfrak{r}(e)}e,\\
&ve^* = \delta_{v, \mathfrak{r}(e)}e^*, \quad e^*v = \delta_{v, \mathfrak{s}(e)}e^*,\\
&ef^* = \delta_{e,f}\mathfrak{r}(e), \quad \sum\limits_{\mathfrak{s}(e) = v}ee^* = v,\\
&ef = \delta_{\mathfrak{r}(e), \mathfrak{s}(f)}ef, \quad e^*f^* = \delta_{\mathfrak{r}(f), \mathfrak{s}(e)}e^*f^*, \quad ef^* = \delta_{\mathfrak{r}(e),\mathfrak{r}(f)}ef^*,
\end{align*}
here $v,u \in V$ and $e,f \in E$.

\smallskip

\paragraph{Leavitt path algebra $W(\ell)$.} Let $\ell \ge 1$, and let us consider the graph $O_\ell$ with $\ell$ loops (see Fig.\ref{loopgraph}). The correspondence Leavitt path algebra $L(O_\ell)$ is denoted by $W(\ell)$.
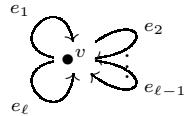
\begin{figure}[h!]
\[
\xymatrix{
\bullet^v \ar@(l,u)^{e_1} \ar@(ur,r)^{e_2}_(.4)\vdots \ar@(r,dr)^{e_{\ell-1}} \ar@(l,d)_{e_\ell}
}
\]
\caption{Here the graph $O_\ell$ is shown. The correspondence Leavitt path algebra is denoted by $W(\ell)$.}\label{loopgraph}
\end{figure}

\smallskip

\par Let us denote by $\Omega$ the set of all paths of the graph $O_\ell$. Then, for the Leavitt path algebra $W(\ell)$, the Theorem \ref{GSBLev} has the following form

\smallskip

\begin{theorem}\label{GSBW}
The following elements form a basis of the Leavitt path algebra $W(\ell)$:
\begin{itemize}
\item[(i)] the vertex $\{v\}$,
\item[(ii)] the set of all paths $\Omega$,
\item[(iii)] the set $\{p^*: p\in \Omega\} = \Omega^*$,
\item[(iv)] the set $\mathfrak{M}$ of words of the form $wh^*$, where $w = e_1\cdots e_n \in \mathfrak{P}$, $h^* = (f_1 \cdots f_m)^* = f_m^* \cdots f_1^* \in \mathfrak{P}^*$, $e_i,f_j \in E$, are paths that end at the same vertex $\mathfrak{r}(w) = \mathfrak{r}(h)$, with the condition that the if edges $e_n$ and $f_m$ are either distinct or equal, but not special.
\end{itemize}
\par The Gr\"obner--Shirshov basis of the Leavitt path algebra $W(\ell)$, can be described as follows:
\begin{align}
&vv=v,\label{GSB1}\\
&ve = ev= e, & e \in E,\\
&ve^* = e^*v = e^*, & e^* \in E^*,\\
&e^*f = \delta_{e,f}v & e,f \in E,\\
&\sum\limits_{i=1}^\ell e_ie_i^* = v, & e_1, \ldots, e_\ell \in E.\label{GSBl}
\end{align}
\end{theorem}

\smallskip

\begin{proposition}
Any derivation $\mathscr{D}$ of the Leavitt path algebra $W(\ell)$ satisfies the following equations:
\begin{align*}
&\mathscr{D}(v)v + v\mathscr{D}(v) = \mathscr{D}(v),\\
&\mathscr{D}(v)e + v\mathscr{D}(e) = \mathscr{D}(e)v + e\mathscr{D}(v) = \mathscr{D}(e), &e \in E,\\
&\mathscr{D}(v)e^* + v\mathscr{D}(e^*) = \mathscr{D}(e^*)v + e^* \mathscr{D}(v) = \mathscr{D}(e^*), & e^* \in E^*,\\
&\mathscr{D}(e^*)f + e^*\mathscr{D}(f) = \delta_{e,f}\mathscr{D}(v), & e,f \in E,\\
&\sum\limits_{i=1}^\ell\Bigl(\mathscr{D}(e_i)e_i^* + e_i\mathscr{D}(e_i^*)\Bigr) = \mathscr{D}(v), & e_1,\ldots,e_\ell\in E.
\end{align*}
\end{proposition}
\begin{proof}
It immediately follows from the definition of derivation and the Theorem \ref{GSBW}.
\end{proof}

\smallskip

\begin{remark}
Without loss of generality we have assumed that $e_1$ is the special edge.
\end{remark}

\section{Derivations of the Leavitt path algebra}

Since the basis of the Leavitt path algebra $W(\ell)$ has been described, then any value $\mathscr{D}(x)$ of any linear map $\mathscr{D}: W(\ell) \to W(\ell)$, can be presented as follows
\[
\mathscr{D}(x) = \alpha_v(x)v + \sum\limits_{p \in \Omega}\Bigl(\beta_p(x)p+ \gamma_p(x)p^* \Bigr) + \sum\limits_{wh^* \in \mathfrak{M}}\rho_{wh^*}(x)wh^*.
\]
Here $x \in \{v\}\cup E \cup E^*$, and almost all scalars $\alpha,\beta,\gamma,\rho \in R$ are zero.

\smallskip

\begin{proposition} Let $\mathscr{D}$ be a derivation of Leavitt path algebra $W(\ell)$. Then
$\mathscr{D}(v) = 0.$
\end{proposition}
\begin{proof}
Using the equation $\mathscr{D}(v)v + v \mathscr{D}(v) = \mathscr{D}(v)$, we get
\begin{eqnarray*}
\mathscr{D}(v)v &=& \left(\alpha_v(v)v + \sum\limits_{p \in \Omega}\Bigl(\beta_p(v)p+ \gamma_p(v)p^* \Bigr) + \sum\limits_{wh^* \in \mathfrak{M}}\rho_{wh^*}(v)wh^* \right)v = \\&& = \alpha_v(v)v + \sum\limits_{p \in \Omega}\Bigl(\beta_p(v)p+ \gamma_p(v)p^* \Bigr) + \sum\limits_{wh^* \in \mathfrak{M}}\rho_{wh^*}(v)wh^*
\end{eqnarray*}
and
\begin{eqnarray*}
v\mathscr{D}(v) &=& v\left(\alpha_v(v)v + \sum\limits_{p \in \Omega}\Bigl(\beta_p(v)p+ \gamma_p(v)p^* \Bigr) + \sum\limits_{wh^* \in \mathfrak{M}}\rho_{wh^*}(v)wh^* \right) = \\&& = \alpha_v(v)v + \sum\limits_{p \in \Omega}\Bigl(\beta_p(v)p+ \gamma_p(v)p^* \Bigr) + \sum\limits_{wh^* \in \mathfrak{M}}\rho_{wh^*}(v)wh^*.
\end{eqnarray*}
It follows $\mathscr{D}(v) = 0$, as claimed.
\end{proof}

\smallskip

\par It follows that we can rewrite all equations determining the derivation in the Leavitt path algebra in the following form

\begin{align*}
&v\mathscr{D}(e)v = \mathscr{D}(e)\\
&v\mathscr{D}(e^*)v = \mathscr{D}(e^*)\\
&\mathscr{D}(e_i^*)e_j+e_i^*\mathscr{D}(e_j) = 0, \quad 1 \le i,j \le \ell\\
&\sum\limits_{i=1}^\ell\Bigl(\mathscr{D}(e_i)e^*_i+e_i\mathscr{D}(e_i^*)\Bigr) = 0.
\end{align*}

\smallskip

\begin{theorem}\label{genth}
Any derivation $\mathscr{D}$ of the Leavitt path algebra $W(\ell)$ can be described as
\[
\mathscr{D}(x) = \begin{cases}0, \mbox{ if $x =v$},\\
\alpha_v(x)v + \sum\limits_{p\in\Omega}\Bigl(\beta_p(x)p+\gamma_p(x)p^*\Bigr)+\sum\limits_{wh^*\in\mathfrak{M}}\rho_{wh^*}(x)wh^*,\mbox{ if $x \in E\cup E^*$,}
\end{cases}
\]
where almost all scalars $\alpha(x),\beta(x), \gamma(x), \rho(x)\in R$ are zero, and they satisfy the following equations
\begin{align}
&\gamma_{e_j}(e_i^*) + \beta_{e_i}(e_j) = 0,\label{e*e1}\\
&\beta_p(e_i^*) + (1-\delta_{1,j})\rho_{pe_je_j^*}(e_i^*) + \beta_{e_ipe_j}(e_j) = 0,& p\in \Omega,\label{e*e2}\\
&\rho_{pe_j^*}(e_i^*) + \beta_{e_ip}(e_j) = 0, & p\in\Omega,\,p_z \ne e_j, \label{e*e3}\\
&\alpha_v(e_i^*) +(1-\delta_{1,j})\rho_{e_je_j^*}(e_i^*) + \beta_{e_ie_j}(e_j) = 0,\label{e*e4}\\
&\gamma_{e_jpe_i}(e_i^*) + \gamma_p(e_j) + (1-\delta_{1,i})\rho_{e_i(pe_i)^*}(e_j) = 0, & p\in \Omega,\label{e*e5}\\
& \gamma_{e_jp}(e_i^*)+\rho_{e_ip^*}(e_j)=0,& p\in\Omega,\,p_z \ne e_i,\label{e*e6}\\
&\alpha_v(e_j)+\gamma_{e_je_i}(e_i^*) + (1-\delta_{1,i})\rho_{e_ie_i^*}(e_j) = 0,\label{e*e7}\\
&\rho_{w(e_jh)^*}(e_i^*) + \rho_{e_iwh^*}(e_j)=0,& wh^* \in\mathfrak{M},\label{e*e8}
\end{align}
for any $1\le i,j\le\ell$.
\end{theorem}
\begin{proof}
We have
\begin{eqnarray*}
v\mathscr{D}(e)v = \mathscr{D}(e)  &=& v \left( \alpha_v(e)v + \sum\limits_{p \in \Omega}\Bigl(\beta_p(e)p+ \gamma_p(e)p^* \Bigr) + \sum\limits_{wh^* \in \mathfrak{M}}\rho_{wh^*}(e)wh^*\right)v = \\&&=\alpha_v(e)+\sum\limits_{p \in \Omega}\Bigl(\beta_p(e)p+ \gamma_p(e)p^* \Bigr) + \sum\limits_{wh^* \in \mathfrak{M}}\rho_{wh^*}(e)wh^*.
\end{eqnarray*}

\smallskip

\par Further,
\begin{eqnarray*}
v\mathscr{D}(e^*)v = \mathscr{D}(e^*)  &=& v \left( \alpha_v(e^*)v + \sum\limits_{p \in \Omega}\Bigl(\beta_p(e^*)p+ \gamma_p(e^*)p^* \Bigr) + \sum\limits_{wh^* \in \mathfrak{M}}\rho_{wh^*}(e^*)wh^*\right)v = \\&&=\alpha_v(e)v+\sum\limits_{p \in \Omega}\Bigl(\beta_p(e^*)p+ \gamma_p(e^*)p^* \Bigr) + \sum\limits_{wh^* \in \mathfrak{M}}\rho_{wh^*}(e^*)wh^*.
\end{eqnarray*}

\smallskip

\par Let us consider for any $1 \le i,j \le \ell$ the equations $\mathscr{D}(e_i^*)e_j+e_i^*\mathscr{D}(e_j) = 0$. We have:
\begin{eqnarray*}
\mathscr{D}(e_i^*)e_j &=& \left(\alpha_v(e_i^*)v+\sum\limits_{p \in \Omega}\Bigl(\beta_p(e_i^*)p+ \gamma_p(e_i^*)p^* \Bigr) + \sum\limits_{wh^* \in \mathfrak{M}}\rho_{wh^*}(e_i^*)wh^*\right)e_j =\\& =& \alpha_v(e_i^*)e_j+ \sum\limits_{p \in \Omega}\beta_p(e_i^*)pe_j + \sum\limits_{p \in \Omega}\gamma_{e_jp}(e_i^*)p^* + \gamma_{e_j}(e_i^*)v + \\&&+ \sum\limits_{we_j^* \in \mathfrak{M}}\rho_{we_j^*}(e_i^*)w + \sum\limits_{wh^*\in\mathfrak{M}}\rho_{w(e_jh)^*}(e_i^*)wh^*,
\end{eqnarray*}
and
\begin{eqnarray*}
e_i^*\mathscr{D}(e_j) &=& e_i^*\left(\alpha_v(e_j)v+\sum\limits_{p \in \Omega}\Bigl(\beta_p(e_j)p+ \gamma_p(e_j)p^* \Bigr) + \sum\limits_{wh^* \in \mathfrak{M}}\rho_{wh^*}(e_j)wh^*\right) =\\& =&\alpha_v(e_j)e_i^*+\beta_{e_i}(e_j)v + \sum\limits_{p \in \Omega}\beta_{e_ip}(e_j)p + \sum\limits_{p \in \Omega}\gamma_p(e_j)(pe_i)^* + \\&&+ \sum\limits_{e_ih^*\in \mathfrak{M}}\rho_{e_ih^*}(e_j)h^* + \sum\limits_{wh^*\in\mathfrak{M}}\rho_{e_iwh^*}(e_j)wh^*.
\end{eqnarray*}

\smallskip

\par Let us add up similar terms:
\begin{eqnarray*}
\Bigl.\mathscr{D}(e_i^*)e_j + e_i^* \mathscr{D}(e_j)\Bigr|_{\{v\}} & = & \gamma_{e_j}(e_i^*)v + \beta_{e_i}(e_j)v,
\end{eqnarray*}
\begin{eqnarray*}
\Bigl.\mathscr{D}(e_i^*)e_j + e_i^* \mathscr{D}(e_j)\Bigr|_\Omega & = &\alpha_v(e_i^*)e_j+ \sum\limits_{p \in \Omega}\beta_p(e_i^*)pe_j + \sum\limits_{we_j^* \in \mathfrak{M}}\rho_{we_j^*}(e_i^*)w +\sum\limits_{p \in \Omega}\beta_{e_ip}(e_j)p =\\ &=& \sum\limits_{p \in \Omega}\beta_p(e_i^*)pe_j + \sum\limits_{pe_je_j^* \in \mathfrak{M}}\rho_{pe_je_j^*}(e_i^*)pe_j + \sum\limits_{we_j^* \in \mathfrak{M},\,w_z \ne e_j}\rho_{we_j^*}(e_i^*)w+\\&& +\alpha_v(e_i^*)e_j +(1-\delta_{1,j})\rho_{e_je_j^*}(e_i^*)e_j + \beta_{e_ie_j}(e_j)e_j+\\ &&+ \sum\limits_{p \in \Omega}\beta_{e_ipe_j}(e_j)pe_j + \sum\limits_{p\in\Omega,\,p_z \ne e_j}\beta_{e_ip}(e_j)p=\\ &=& \sum\limits_{p\in \Omega} \Bigl(\beta_{p}(e_i^*) + (1-\delta_{1,j})\rho_{pe_je_j^*}(e_i^*)+\beta_{e_ipe_j}(e_j) \Bigr)pe_j+\\&&+\sum\limits_{p\in\Omega,\,p_z \ne e_j}\Bigl( \rho_{pe_j^*}(e_i^*) + \beta_{e_ip}(e_j)\Bigr)p + \\&& + \Bigl(\alpha_v(e_i^*) +(1-\delta_{1,j})\rho_{e_je_j^*}(e_i^*) + \beta_{e_ie_j}(e_j)\Bigr)e_j,
\end{eqnarray*}
\begin{eqnarray*}
\Bigl.\mathscr{D}(e_i^*)e_j + e_i^* \mathscr{D}(e_j)\Bigr|_{\Omega^*} & = &\alpha_v(e_j)e_i^*+ \sum\limits_{p\in \Omega}\gamma_{e_jp}(e_i^*)p^* + \sum\limits_{p \in \Omega}\gamma_p(e_i)(pe_j)^* + \sum\limits_{e_ih^* \in \mathfrak{M}}\rho_{e_ih^*}(e_j)h^* = \\ &=&\alpha_v(e_j)e_i^*+\gamma_{e_je_i}(e_i^*)e_i^*+\sum\limits_{p \in \Omega}\gamma_{e_jpe_i}(e_i^*)(pe_i)^* + \sum\limits_{p \in \Omega,\, p_z \ne e_i}\gamma_{e_jp}(e_i^*)p^*+\\&&+\sum\limits_{p \in \Omega}\gamma_p(e_i)(pe_i)^* + (1-\delta_{1,i})\rho_{e_ie_i^*}(e_j)e_i^*+\sum\limits_{e_ie_i^*p^*\in\mathfrak{M}}\rho_{e_i(pe_i)^*}(e_j)(pe_i)^*+\\ &&+\sum\limits_{e_ih^*\in\mathfrak{M},\,h_z \ne e_i}\rho_{eh^*}(e_j)h^*=\\&=&\sum\limits_{p\in \Omega}\Bigl(\gamma_{e_jpe_i}(e_i^*) + \gamma_p(e_i) + (1-\delta_{i,1})\rho_{e_i(pe_i)^*}(e_j)\Bigr)(pe_i)^* +\\&&+ \sum\limits_{p \in \Omega,\,p_z \ne e_i}\Bigl( \gamma_{e_jp}(e_i^*)+\rho_{e_ip^*}(e_j)\Bigr)p^*+\\&&+\Bigl(\alpha_v(e_j)+\gamma_{e_je_i}(e_i^*) + (1-\delta_{1,i})\rho_{e_ie_i^*}(e_j) \Bigr)e_i^*,
\end{eqnarray*}
\begin{eqnarray*}
\Bigl.\mathscr{D}(e_i^*)e_j + e_i^* \mathscr{D}(e_j)\Bigr|_\mathfrak{M} & = & \sum\limits_{wh^*\in\mathfrak{M}}\rho_{w(e_jh)^*}(e_i^*)wh^* + \sum\limits_{wh^*\in\mathfrak{M}}\rho_{e_iwh^*}(e_j)wh^* = \\&=&\sum\limits_{wh^* \in \mathfrak{M}}\Bigl(\rho_{w(e_jh)^*}(e_i^*) + \rho_{e_iwh^*}(e_j) \Bigr)wh^*.
\end{eqnarray*}

\smallskip

\par So, for any $1 \le i,j\le \ell$, we have the following equations:
\begin{align*}
&\gamma_{e_j}(e_i^*) + \beta_{e_i}(e_j) = 0,\\
&\beta_p(e_i^*) + (1-\delta_{1,j})\rho_{pe_je_j^*}(e_i^*) + \beta_{e_ipe_j}(e_j) = 0,\qquad p\in \Omega,\\
&\rho_{pe_j^*}(e_i^*) + \beta_{e_ip}(e_j) = 0, \qquad p\in\Omega,\,p_z \ne e_j,\\
&\alpha_v(e_i^*) +(1-\delta_{1,j})\rho_{e_je_j^*}(e_i^*) + \beta_{e_ie_j}(e_j) = 0,\\
&\gamma_{e_jpe_i}(e_i^*) + \gamma_p(e_j) + (1-\delta_{1,i})\rho_{e_i(pe_i)^*}(e_j) = 0, \qquad p\in \Omega,\\
& \gamma_{e_jp}(e_i^*)+\rho_{e_ip^*}(e_j)=0,\qquad p\in\Omega,\,p_z \ne e_i,\\
&\alpha_v(e_j)+\gamma_{e_je_i}(e_i^*) + (1-\delta_{1,i})\rho_{e_ie_i^*}(e_j) = 0,\\
&\rho_{w(e_jh)^*}(e_i^*) + \rho_{e_iwh^*}(e_j)=0,\qquad wh^* \in\mathfrak{M}.
\end{align*}

\smallskip

\par On other hand, let us consider the equation $\sum\limits_{r=1}^\ell\Bigl(\mathscr{D}(e_r)e_r^* + e_r \mathscr{D}(e_r^*)\Bigr) = 0$. For the special edge $e_1$, we have:
\begin{eqnarray*}
\mathscr{D}(e_1)e_1^* &=& \left(\alpha_v(e_1)v+\sum\limits_{p \in \Omega}\Bigl(\beta_p(e_1)p+ \gamma_p(e_1)p^* \Bigr) + \sum\limits_{wh^* \in \mathfrak{M}}\rho_{wh^*}(e_1)wh^*\right)e_1^*=\\&=& \alpha_v(e_1)e_1^*+ \beta_{e_1}(e_1)v - \sum\limits_{k=2}^\ell\beta_{e_1}(e_1)e_ke_k^* + \sum\limits_{p\in\Omega}\beta_{pe_1}(e_1)p - \sum\limits_{p\in \Omega} \sum\limits_{k=2}^\ell\beta_{pe_1}(e_1)pe_ke_k^* +\\&&+ \sum\limits_{p\in \Omega,\,p_z \ne e_1}\beta_p(e_1)pe_1^*+\sum\limits_{p \in \Omega}\gamma_p(e_1)(e_1p)^* + \sum\limits_{wh^* \in \mathfrak{M}}\rho_{wh^*}(e_1)w(e_1h)^*,
\end{eqnarray*}
and
\begin{eqnarray*}
e_1\mathscr{D}(e_1^*) &=&e_1\left(\alpha_v(e_1^*)v+\sum\limits_{p \in \Omega}\Bigl(\beta_p(e_1^*)p+ \gamma_p(e_1^*)p^* \Bigr) + \sum\limits_{wh^* \in \mathfrak{M}}\rho_{wh^*}(e_1^*)wh^*\right) =\\ &=&\alpha_v(e_1^*) e_1 +\sum\limits_{p\in \Omega}\beta_p(e_1^*)e_1p + \gamma_{e_1}(e_1^*)v - \sum\limits_{k=2}^\ell\gamma_{e_1}(e_1^*)e_ke_k^* + \sum\limits_{p\in \Omega}\gamma_{pe_1}(e_1^*)p^* - \\&&-\sum\limits_{p\in \Omega}\sum\limits_{k=2}^\ell\gamma_{pe_1}(e_1^*)e_ke_k^*p^*+\sum\limits_{p\in\Omega,\,p_z \ne e_1}\gamma_p(e_1^*)e_1p^* + \sum\limits_{wh^*\in \mathfrak{M}}\rho_{wh^*}(e_1^*)e_1wh^*.
\end{eqnarray*}

\smallskip

\par Further, we have:
\begin{eqnarray*}
\mathscr{D}(e_r)e_r^* &=& \left(\alpha_v(e_r)v+\sum\limits_{p \in \Omega^*}\Bigl(\beta_p(e_r)p+ \gamma_p(e_r)p^* \Bigr) + \sum\limits_{wh^* \in \mathfrak{M}}\rho_{wh^*}(e_r)wh^*\right)e_r^*=\\&=&\alpha_v(e_r)e_r^*+\sum\limits_{p\in \Omega}\beta_p(e_r)pe_r^* + \sum\limits_{p\in \Omega}\gamma_p(e_r)(e_rp)^* + \sum\limits_{wh^*\in\mathfrak{M}}\rho_{wh^*}(e_r)w(e_rh)^*,
\end{eqnarray*}
and
\begin{eqnarray*}
e_r\mathscr{D}(e_r^*) &=&e_r\left(\alpha_v(e_r^*)v+\sum\limits_{p \in \Omega}\Bigl(\beta_p(e_r^*)p+ \gamma_p(e_r^*)p^* \Bigr) + \sum\limits_{wh^* \in \mathfrak{M}}\rho_{wh^*}(e_r^*)wh^*\right) =\\&=& \alpha_v(e_r^*)e_r+ \sum\limits_{p \in \Omega}\beta_p(e_r^*)e_rp + \sum\limits_{p \in \Omega}\gamma_p(e_r^*)e_rp^* + \sum\limits_{wh^* \in \mathfrak{M}}\rho_{wh^*}(e_r^*)e_rwh^*.
\end{eqnarray*}

\par We get:
\begin{eqnarray*}
\left.\sum\limits_{r=1}^\ell\Bigl(\mathscr{D}(e_r)e_r^* + e_r \mathscr{D}(e_r^*)\Bigr)\right|_{\{v\}} &=& \beta_{e_1}(e_1)v + \gamma_{e_1}(e_1^*)v = 0.
\end{eqnarray*}

\smallskip

\par It follows $\beta_{e_1}(e_1)v + \gamma_{e_1}(e_1^*) = 0$, it is the equation (\ref{e*e1}). Further,

\smallskip

\begin{eqnarray*}
\left.\sum\limits_{r=1}^\ell\Bigl(\mathscr{D}(e_r)e_r^* + e_r \mathscr{D}(e_r^*)\Bigr)\right|_\Omega &=&\sum\limits_{r=1}^\ell\alpha_v(e_r^*)e_r+\sum\limits_{p\in\Omega}\beta_{pe_1}(e_1)p +\sum\limits_{r=1}^\ell\sum\limits_{p\in\Omega}\beta_p(e_r^*)e_rp=\\&=&\sum\limits_{r=1}^\ell\Bigl(\alpha_v(e_r^*) +\beta_{e_re_1}(e_1)\Bigr)e_r + \sum\limits_{r=1}^\ell\sum\limits_{p \in \Omega}\Bigl(\beta_{e_rpe_1}(e_1) + \beta_p(e_r^*) \Bigr)e_rp=\\&=& 0.
\end{eqnarray*}

\smallskip

\par It follows
\[
\alpha_v(e_r^*)+\beta_{e_re_1}(e_1) = 0, \qquad \beta_{e_rpe_1}(e_1) + \beta_p(e_r^*)=0,
\]
these equations have been already found (see (\ref{e*e4}) and (\ref{e*e2}) respectively). Further,

\smallskip

\begin{eqnarray*}
\left.\sum\limits_{r=1}^\ell\Bigl(\mathscr{D}(e_r)e_r^* + e_r \mathscr{D}(e_r^*)\Bigr)\right|_{\Omega^*} &=&\sum\limits_{r=1}^\ell\alpha_v(e_r)e_r^*+ \sum\limits_{p \in \Omega}\gamma_{pe_1}(e_1^*)p^* + \sum\limits_{r=1}^\ell\sum\limits_{p\in\Omega}\gamma_p(e_r)(e_rp)^*=\\&=& \sum\limits_{r=1}^\ell\Bigl(\alpha_v(e_r)+\gamma_{e_re_1}(e_1^*)\Bigr)e_r^* + \sum\limits_{r=1}^\ell\sum\limits_{p\in\Omega}\Bigl( \gamma_{e_rpe_1}(e_1^*) + \gamma_{p}(e_r)\Bigr)(e_rp)^*=\\&=&0.
\end{eqnarray*}

\smallskip

\par Consequently, we get:
\[
\alpha_v(e_r)+\gamma_{e_re_1}(e_1^*)e_r^* =0,\qquad \gamma_{e_rpe_1}(e_1^*) + \gamma_{p}(e_r) = 0,
\]
these are the equations (\ref{e*e7}) and (\ref{e*e5}) respectively.

\smallskip

\par So, we have:

\begin{eqnarray*}
\left.\sum\limits_{r=1}^\ell\Bigl(\mathscr{D}(e_r)e_r^* + e_r \mathscr{D}(e_r^*)\Bigr)\right|_{\mathfrak{M}} &=& - \sum\limits_{k=2}^\ell\beta_{e_1}(e_1)e_ke_k^* - \sum\limits_{p\in\Omega}\sum\limits_{k=2}^\ell\beta_{pe_1}(e_1)pe_ke_k^* + \sum\limits_{p\in\Omega,\,p_z \ne e_1}\beta_p(e_1)pe_1^* + \\ &&+ \sum\limits_{wh^*\in\mathfrak{M}}\rho_{wh^*}(e_1)w(e_1h)^* -\sum\limits_{k=2}^\ell\gamma_{e_1}(e_1^*)e_ke_k^* - \sum\limits_{p\in \Omega}\sum\limits_{k=2}^\ell\gamma_{pe_1}(e_1^*)e_ke_k^*p^* + \\&&+\sum\limits_{p\in \Omega,\,p_z \ne e_1}\gamma_p(e_1^*)e_1p^* + \sum\limits_{wh^*\in\mathfrak{M}}\rho_{wh^*}(e_1^*)e_1wh^* + \sum\limits_{r=2}^\ell\sum\limits_{p\in\Omega}\beta_p(e_r)pe_r^* +\\&&+\sum\limits_{r=2}^\ell\sum\limits_{wh^*\in\mathfrak{M}}\rho_{wh^*}(e_r)w(e_rh)^* + \sum\limits_{r=2}^\ell\sum\limits_{p\in\Omega}\gamma_{p}(e_r^*)e_rp^* + \sum\limits_{r=2}^\ell\sum\limits_{wh^*\in\mathfrak{M}}\rho_{wh^*}(e_r^*)e_rwh^*.
\end{eqnarray*}

\smallskip

\par Let us add up similar terms:
\begin{eqnarray*}
\left.\sum\limits_{r=1}^\ell\Bigl(\mathscr{D}(e_r)e_r^* + e_r \mathscr{D}(e_r^*)\Bigr)\right|_{\mathfrak{M}} &=&-\sum\limits_{k=2}^\ell\Bigl(\beta_{e_1}(e_1) + \gamma_{e_1}(e_1^*)\Bigr)e_ke_k^*+ \sum\limits_{r=1}^\ell\sum\limits_{pe_r^* \in\mathfrak{M}}\beta_p(e_r)pe_r^* +
\sum\limits_{r=1}^\ell\sum\limits_{e_rp^*\in\mathfrak{M}}\gamma_p(e_r^*)e_rp^*-\\ &&-\sum\limits_{k=2}^\ell\sum\limits_{p\in\Omega}\beta_{pe_1}(e_1)pe_ke_k^* -\sum\limits_{k=2}^\ell\sum\limits_{p\in\Omega}\gamma_{pe_1}(e_1^*)e_k(pe_k)^*+ \\&&+
\sum\limits_{r=1}^\ell\sum\limits_{wh^*\in\mathfrak{M}}\rho_{wh^*}(e_r)w(e_rh)^* +
\sum\limits_{r=1}^\ell\sum\limits_{wh^*\in\mathfrak{M}}\rho_{wh^*}(e_r^*)e_rwh^* =\\&=&
\sum\limits_{k=2}^\ell\Bigl(-\beta_{e_1}(e_1) - \gamma_{e_1}(e_1) + \beta_{e_k}(e_k) + \gamma_{e_k}(e_k^*)\Bigr)e_ke_k^* + \\&&+\sum\limits_{r=1}^\ell\sum\limits_{k=1,k \ne r}^\ell\Bigl(\beta_{e_k}(e_r) + \gamma_{e_r}(e_k^*)\Bigr)e_ke_r^* + \sum\limits_{r=1}^\ell\sum\limits_{pe_r^*\in\mathfrak{M},\,p\notin E}\beta_p(e_r)pe_r^* + \\&&+\sum\limits_{r=1}^\ell\sum\limits_{e_rp^*\in\mathfrak{M},\,p\notin E}\gamma_p(e_r^*)e_rp^*-\sum\limits_{k=2}^\ell\sum\limits_{p\in\Omega}\beta_{pe_1}(e_1)pe_ke_k^* -\sum\limits_{k=2}^\ell\sum\limits_{p\in\Omega}\gamma_{pe_1}(e_1^*)e_k(pe_k)^*+ \\&&+
\sum\limits_{k=1}^\ell\sum\limits_{r=1}^\ell\sum\limits_{wh^*\in \mathfrak{M}}\Bigl(\rho_{e_kwh^*}(e_r) + \rho_{w(e_rh)^*}(e_k^*)\Bigr)e_kw(e_rh)^*+\\&&+
\sum\limits_{k=1}^\ell\sum\limits_{r=1}^\ell\sum\limits_{e_kh^*\in\mathfrak{M}}\rho_{e_kh^*}(e_r)e_k(e_rh)^* +
\sum\limits_{k=1}^\ell\sum\limits_{r=1}^\ell\sum\limits_{we_k^*\in\mathfrak{M}}\rho_{we_k^*}(e_r^*)e_rwe_k^* =\\&=&0.
\end{eqnarray*}

\smallskip

\par Let us consider the following sum:
\begin{eqnarray*}
S &=& \sum\limits_{r=1}^\ell\sum\limits_{pe_r^*\in\mathfrak{M},\,p\notin E}\beta_p(e_r)pe_r^* +\sum\limits_{r=1}^\ell\sum\limits_{e_rp^*\in\mathfrak{M},\,p\notin E}\gamma_p(e_r^*)e_rp^*-\\&&-\sum\limits_{k=2}^\ell\sum\limits_{p\in\Omega}\beta_{pe_1}(e_1)pe_ke_k^* -\sum\limits_{k=2}^\ell\sum\limits_{p\in\Omega}\gamma_{pe_1}(e_1^*)e_k(pe_k)^*+\\&&+
\sum\limits_{k=1}^\ell\sum\limits_{r=1}^\ell\sum\limits_{e_kh^*\in\mathfrak{M}}\rho_{e_kh^*}(e_r)e_k(e_rh)^* +
\sum\limits_{k=1}^\ell\sum\limits_{r=1}^\ell\sum\limits_{we_k^*\in\mathfrak{M}}\rho_{we_k^*}(e_r^*)e_rwe_k^*.
\end{eqnarray*}

\smallskip

\par We have:
\begin{eqnarray*}
S & = & \sum\limits_{r=1}^\ell\sum\limits_{pe_r^*\in\mathfrak{M},p\notin E}\beta_p(e_r)pe_r^* -\sum\limits_{k=2}^\ell\sum\limits_{p\in\Omega}\beta_{pe_1}(e_1)pe_ke_k^* +\\&&+\sum\limits_{k=1}^\ell\sum\limits_{r=1}^\ell\sum\limits_{we_k^*\in\mathfrak{M}}\rho_{we_k^*}(e_r^*)e_rwe_k^* + \sum\limits_{r=1}^\ell\sum\limits_{e_rp^*\in\mathfrak{M},\,p\notin E} \gamma_p(e_r^*)e_rp^* - \\&&- \sum\limits_{k=2}^\ell\sum\limits_{p\in\Omega}\gamma_{pe_1}(e_1^*)e_k(pe_k)^* + \sum\limits_{k=1}^\ell\sum\limits_{r=1}^\ell\sum\limits_{e_kh^*\in\mathfrak{M}}\rho_{e_kh^*}(e_r)e_k(e_rh)^*= \\ &=& \sum\limits_{r=1}^\ell\sum\limits_{k=2}^\ell\Bigl(\beta_{e_re_k}(e_k) - \beta_{e_re_1}(e_1) + \rho_{e_ke_k^*}(e_r)\Bigr)e_re_ke_k^* +\sum\limits_{k=1}^\ell\sum\limits_{r=1}^\ell\sum\limits_{p\in\Omega,\,p_z \ne e_k}\Bigl(\beta_{e_rp}(e_k) + \rho_{pe_k^*}(e_r^*)\Bigr)e_rpe_k^*+ \\&&+ \sum\limits_{r=1}^\ell\sum\limits_{k=2}^\ell\sum\limits_{p\in\Omega}\Bigl(\beta_{e_rpe_k}(e_k) - \beta_{e_rpe_1}(e_1) + \rho_{pe_ke_k^*}(e_r^*)\Bigr)e_rpe_ke_k^* +\\&&+
\sum\limits_{r=1}^\ell\sum\limits_{k=2}^\ell\Bigl(\gamma_{e_re_k}(e_k^*) -\gamma_{e_re_1}(e_1^*) + \rho_{e_ke_k^*}(e_r)\Bigr)e_k(e_re_k)^* + \sum\limits_{r=1}^\ell\sum\limits_{k=1}^\ell\sum\limits_{p\in\Omega,\,p_z \ne e_k}\Bigl(\gamma_{e_rp}(e_k^*) + \rho_{e_kp^*}(e_r)\Bigr)e_k(e_rp)^*+\\&&+
\sum\limits_{r=1}^\ell\sum\limits_{k=2}^\ell\sum\limits_{p\in\Omega}\Bigl(\gamma_{e_rpe_k}(e_k^*) - \gamma_{e_rp}(e_1^*) + \rho_{e_k(pe_k)^*}(e_r)\Bigr)e_k(e_rpe_k)^*.
\end{eqnarray*}

\smallskip

\par Therefore:
\begin{eqnarray*}
\left.\sum\limits_{r=1}^\ell\Bigl(\mathscr{D}(e_r)e_r^* + e_r \mathscr{D}(e_r^*)\Bigr)\right|_{\mathfrak{M}} &=& \sum\limits_{k=2}^\ell\Bigl(-\beta_{e_1}(e_1) - \gamma_{e_1}(e_1) + \beta_{e_k}(e_k) + \gamma_{e_k}(e_k^*)\Bigr)e_ke_k^* + \\&&+\sum\limits_{r=1}^\ell\sum\limits_{k=1,k \ne r}^\ell\Bigl(\beta_{e_k}(e_r) + \gamma_{e_r}(e_k^*)\Bigr)e_ke_r^*+\\&& +\sum\limits_{r=1}^\ell\sum\limits_{k=2}^\ell\Bigl(\beta_{e_re_k}(e_k) - \beta_{e_re_1}(e_1) + \rho_{e_ke_k^*}(e_r)\Bigr)e_re_ke_k^* +\\&& +\sum\limits_{k=1}^\ell\sum\limits_{r=1}^\ell\sum\limits_{p\in\Omega,\,p_z \ne e_k}\Bigl(\beta_{e_rp}(e_k) + \rho_{pe_k^*}(e_r^*)\Bigr)e_rpe_k^*+ \\&&+ \sum\limits_{r=1}^\ell\sum\limits_{k=2}^\ell\sum\limits_{p\in\Omega}\Bigl(\beta_{e_rpe_k}(e_k) - \beta_{e_rpe_1}(e_1) + \rho_{pe_ke_k^*}(e_r^*)\Bigr)e_rpe_ke_k^* +\\&&+
\sum\limits_{r=1}^\ell\sum\limits_{k=2}^\ell\Bigl(\gamma_{e_re_k}(e_k^*) -\gamma_{e_re_1}(e_1^*) + \rho_{e_ke_k^*}(e_r)\Bigr)e_k(e_re_k)^* + \\&&+ \sum\limits_{r=1}^\ell\sum\limits_{k=1}^\ell\sum\limits_{p\in\Omega,\,p_z \ne e_k}\Bigl(\gamma_{e_rp}(e_k^*) + \rho_{e_kp^*}(e_r)\Bigr)e_k(e_rp)^*+\\&&+
\sum\limits_{r=1}^\ell\sum\limits_{k=2}^\ell\sum\limits_{p\in\Omega}\Bigl(\gamma_{e_rpe_k}(e_k^*) - \gamma_{e_rp}(e_1^*) + \rho_{e_k(pe_k)^*}(e_r)\Bigr)e_k(e_rpe_k)^*+\\&&+\sum\limits_{k=1}^\ell\sum\limits_{r=1}^\ell\sum\limits_{wh^*\in \mathfrak{M}}\Bigl(\rho_{e_kwh^*}(e_r) + \rho_{w(e_rh)^*}(e_k^*)\Bigr)e_kw(e_rh)^* =\\&=&0.
\end{eqnarray*}

\smallskip

\par (1) The equations
\[
-\beta_{e_1}(e_1) - \gamma_{e_1}(e_1) + \beta_{e_k}(e_k) + \gamma_{e_k}(e_k^*)=0, \quad 2 \le k \le \ell,
\]
are the equations (\ref{e*e1}).

\smallskip

\par (2) The equations
\[
\beta_{e_k}(e_r) + \gamma_{e_r}(e_k^*) = 0, \qquad 1 \le k \ne r \le \ell,
\]
are the equations (\ref{e*e1}).

\smallskip

\par (3) Let us consider the equation
\[
\beta_{e_re_k}(e_k) - \beta_{e_re_1}(e_1) - \rho_{e_ke_k^*}(e_r) = 0, \qquad 1 \le k \ne r \le \ell,
\]
using (\ref{e*e4}) we get
\[
\beta_{e_re_k}(e_k) - \beta_{e_re_1}(e_1) - \rho_{e_ke_k^*}(e_r) = \beta_{e_re_k}(e_k) + \alpha_v(e_r^*) -\rho_{e_ke_k^*}(e_r) = 0,
\]
and we get the equations (\ref{e*e4}).

\smallskip

\par (4) The equations
\[
\beta_{e_rp}(e_k) + \rho_{pe_k^*}(e_r^*) = 0, \qquad 1 \le k,r \le \ell,\, p_z \ne e_k,\, p \in\Omega,
\]
are the equations (\ref{e*e3}).

\smallskip

\par (5) Let us consider the equations
\[
-\beta_{e_rpe_1}(e_1) +\beta_{e_rpe_k}(e_k) +\rho_{pe_ke_k^*}(e_r^*) = 0,\qquad 1\le r \le \ell, \,2 \le k \le \ell,
\]
using (\ref{e*e2}) we have $\rho_{pe_ke_k^*}(e_r^*) = - \beta_{e_rpe_k}(e_k) - \beta_p(e_r^*)$, then we get
\begin{eqnarray*}
0&=&-\beta_{e_rpe_1}(e_1) +\beta_{e_rpe_k}(e_k) +\rho_{pe_ke_k^*}(e_r^*) =\\
&=&-\beta_{e_rpe_1}(e_1) +\beta_{e_rpe_k}(e_k) - \beta_{e_rpe_k}(e_k) - \beta_p(e_r^*)=\\
&=& -\beta_{e_rpe_1}(e_1)- \beta_p(e_r^*),
\end{eqnarray*}
and we get the equations (\ref{e*e2}).

\smallskip

\par (6) Let us consider the equations
\[
\gamma_{e_re_k}(e_k^*) - \gamma_{e_re_1}(e_1^*) + \rho_{e_k(e_re_k)^*}(e_r) = 0, \qquad 1 \le r \le \ell, \, 2 \le k \le \ell,
\]
using (\ref{e*e7}) we get
\[
\gamma_{e_re_k}(e_k^*) - \gamma_{e_re_1}(e_1^*) + \rho_{e_ke_k^*}(e_r) = \gamma_{e_re_k}(e_k^*) + \alpha_{v}(e_r) + \rho_{e_ke_k^*}(e_r) = 0,
\]
and we get the equations (\ref{e*e7}).

\smallskip

\par (7) The equations
\[
\gamma_{e_rp}(e_k^*) + \rho_{e_kp}(e_r) = 0, \qquad 1 \le r,k\le \ell,\, , p\in\Omega,\,p_z \ne e_z,
\]
are the equations (\ref{e*e6}).

\smallskip

\par (8) Let us consider the equations
\[
-\gamma_{e_rpe_1}(e_1^*) +\gamma_{e_rpe_k}(e_k^*) + \rho_{e_k(pe_k)^*}(e_r) = 0, \qquad 1 \le r \le \ell, \, 2 \le k \le \ell,
\]
using (\ref{e*e5}) we get $\rho_{e_k(pe_k)^*}(e_r) = - \gamma_{e_rpe_k}(e_k^*) - \gamma_p(e_k)$ then we have
\begin{eqnarray*}
0 &=& -\gamma_{e_rpe_1}(e_1^*) +\gamma_{e_rpe_k}(e_k^*) + \rho_{e_k(pe_k)^*}(e_r) = \\ &=& -\gamma_{e_rpe_1}(e_1^*) +\gamma_{e_rpe_k}(e_k^*) - \gamma_{e_rpe_k}(e_k^*) - \gamma_p(e_k) = \\
&=& -\gamma_{e_rpe_1}(e_1^*)  - \gamma_p(e_k),
\end{eqnarray*}
and we get the equations (\ref{e*e5}).

\smallskip

\par (9) The equations
\[
\rho_{e_kwh^*}(e_r) + \rho_{w(e_rh)^*}(e_k^*) = 0, \qquad 1 \le k,r \le \ell,\, wh^*\in\mathfrak{M},
\]
are the equations (\ref{e*e8}).

\smallskip

\par So, we have, for any $1\le i,j\le\ell$, the following equations:
\begin{align*}
&\gamma_{e_j}(e_i^*) + \beta_{e_i}(e_j) = 0,\\
&\beta_p(e_i^*) + (1-\delta_{1,j})\rho_{pe_je_j^*}(e_i^*) + \beta_{e_ipe_j}(e_j) = 0,& p\in \Omega,\\
&\rho_{pe_j^*}(e_i^*) + \beta_{e_ip}(e_j) = 0, & p\in\Omega,\,p_z \ne e_j,\\
&\alpha_v(e_i^*) +(1-\delta_{1,j})\rho_{e_je_j^*}(e_i^*) + \beta_{e_ie_j}(e_j) = 0,\\
&\gamma_{e_jpe_i}(e_i^*) + \gamma_p(e_j) + (1-\delta_{1,i})\rho_{e_i(pe_i)^*}(e_j) = 0, & p\in \Omega,\\
& \gamma_{e_jp}(e_i^*)+\rho_{e_ip^*}(e_j)=0,& p\in\Omega,\,p_z \ne e_i,\\
&\alpha_v(e_j)+\gamma_{e_je_i}(e_i^*) + (1-\delta_{1,i})\rho_{e_ie_i^*}(e_j) = 0,\\
&\rho_{w(e_jh)^*}(e_i^*) + \rho_{e_iwh^*}(e_j)=0,& wh^* \in\mathfrak{M},
\end{align*}
as claimed.
\end{proof}

\smallskip

\begin{example}[Derivations of Laurent polynomial ring]
It is well known that the Leavitt path algebra $W(1)$ is the Laurent polynomial ring, i.e.,
\begin{align*}
&L\Bigl(\xymatrix{\bullet^v \ar@(ul, dl)_e}\Bigr) \cong R[t,t^{-1}],\\
&u\longleftrightarrow 1, \quad e \longleftrightarrow t, \quad e^* \longleftrightarrow t^{-1}.
\end{align*}

\smallskip

\par In this case, using the Theorem \ref{genth} we can describe any derivations $\mathscr{D}$ as follows:
\[
\mathscr{D}(x) = \begin{cases}0, \mbox{ if $x =v$},\\
\alpha_v(x)v + \sum\limits_{i=1}^\infty\Bigl(\beta_{e^i}(x)e^i + \gamma_{e^i}(x)(e^i)^*\Bigr),\mbox{ if $x \in \{e\}\cup \{e^*\}$,}
\end{cases}
\]
where $e^i: = \underbrace{e \cdots e}_i$, and almost all scalars $\alpha(x),\beta(x), \gamma(x)\in R$ are zero, and they satisfy the following equations
\begin{align*}
&\gamma_{e}(e^*) + \beta_{e}(e) = 0,\\
&\beta_{e^i}(e^*) + \beta_{e^{i+2}}(e) = 0,& i \ge 1\\
&\alpha_v(e^*) + \beta_{e^2}(e) = 0,\\
&\gamma_{e^{i+2}}(e^*) + \gamma_{e^i}(e) = 0, & i\ge 1\\
&\alpha_v(e)+\gamma_{e^2}(e^*) =0.
\end{align*}

\smallskip

\par 1) Let $\alpha_v(e) =1$, $\beta_{e^i}(e) =\gamma_{e^i}(e)= 0$ for any $i \ge 1$, then $\gamma_{e^2}(e^*) = -1$ and another scalars are zero. We get the well-known formulas
\[
\mathscr{D}(e)= v \longleftrightarrow \dfrac{\partial t}{\partial t} = 1, \qquad \mathscr{D}(e^*) = - (ee)^* \longleftrightarrow \dfrac{\partial t^{-1}}{\partial t} = -(t^2)^{-1}.
\]

\smallskip

\par 2) Let $\alpha_v(e^*) =1$ and $\beta_{e^i}(e^*) = \gamma_{e^i}(e^*) = 0$ for any $i \ge 1$, then $\beta_{e^2}(e) = 0$ and another scalars are zero. We again get the well-known formulas,
\[
\mathscr{D}'(e) = -ee \longleftrightarrow \dfrac{\partial t}{\partial t^{-1}}= -t^2, \qquad \mathscr{D}'(e^*) = v\longleftrightarrow \dfrac{\partial t^{-1}}{\partial t^{-1}} = 1.
\]
\end{example}

\smallskip

\begin{example}
Let us consider the Leavitt path algebra $W(2)$, and let us consider some of its derivations. Using Theorem \ref{genth}, we can put
\[
\mathscr{D}(e_1) = \alpha_v(e_1)v, \qquad \mathscr{D}(e_2) = \alpha_v(e_2)v,
\]
then we get
\[
\mathscr{D}(e_1^*) = -\alpha_v(e_1)(e_1e_1)^* - \alpha_v(e_2)(e_2e_1)^*, \qquad \mathscr{D}(e_2^*) = -\alpha_v(e_1)(e_1e_2)^* - \alpha_v(e_2)(e_2e_2)^*.
\]
\end{example}

\section{The inner and the outer derivations of $W(\ell)$}
In this section we describe all inner derivations of the Leavitt path algebra $W(\ell)$. We use the standard notations, that is $\mathrm{ad}_x(y): = [x,y] = xy - yx$. From Theorem \ref{GSBLev} we have for any $v\in V$, $p \in \Omega$ and $wh^* \in \mathfrak{M}$ the following basic elements of $\mathrm{Im}(d_0)$,
\begin{eqnarray*}
\mathrm{ad}_v(-) &= &\begin{cases}\mathrm{ad}_v(v) = vv-vv = 0\\
 \mathrm{ad}_v(e) = ve - ev = 0\\ \mathrm{ad}_v(e^*) = ve^* - e^*v = 0
 \end{cases}\\
\mathrm{ad}_p(-) &= &\begin{cases}
 \mathrm{ad}_p(v) = pv - vp=0\\
 \mathrm{ad}_p(e) = pe - ep\\
 \mathrm{ad}_p(e^*) = pe^* - \delta_{p_0,e}(p/p_0)
 \end{cases} \\
\mathrm{ad}_{p^*}(-) &= &\begin{cases}
 \mathrm{ad}_{p^*}(v) = p^*v - vp^* = 0\\
 \mathrm{ad}_{p^*}(e) = \delta_{p_0,e}(p/p_0)^* - ep^*\\
 \mathrm{ad}_{p^*}(e^*) = p^*e^* - e^*p^*
 \end{cases}\\
\mathrm{ad}_{wh^*}(-) &= &\begin{cases}
\mathrm{ad}_{wh^*}(v) = wh^*v - vwh^* = 0\\
\mathrm{ad}_{wh^*}(e) = \delta_{h_0,e}w(h/h_0)^* - (ew)h^*\\
\mathrm{ad}_{wh^*}(e^*) = wh^*e^* - \delta_{w_0, e}(w/w_0)h^*.
\end{cases}
\end{eqnarray*}

\smallskip

\par It follows that any $A\in\mathrm{InnDer}(W(\ell))$, can be presented as follows:
\[
A = \sum\limits_{p \in \Omega}\Bigl(\nu_p\mathrm{ad}_p + \nu_p'\mathrm{ad}_{p^*}\Bigr) + \sum\limits_{wh^* \in \mathfrak{M}}\nu_{wh^*}''\mathrm{ad}_{wh^*}.
\]

\smallskip

\begin{theorem}\label{innerder}
Any inner derivation $A$ of the Leavitt path algebra $W(\ell)$ can be described as follows
\[
A(x) = \begin{cases} 0, \mbox{ if } x =v,\\ \alpha_v(x)v+ \sum\limits_{p \in \Omega}\Bigl(\beta_p(x)p+\gamma_p(x)p^*\Bigr) + \sum\limits_{wh^*\in\mathfrak{M}}\rho_{wh^*}(x)wh^*, \mbox{ if $x \in E \cup E^*$,}\end{cases}
\]
where almost all scalers $\beta(x), \gamma(x), \rho(x)$ are zero and they satisfy the following equations,
\begin{align*}
&\beta_{e_1pe_1}(e_1) = \beta_{e_1pe_1}(e_1^*)=0, & p\in \Omega,\\
&\gamma_{e_1pe_1}(e_1^*) = \gamma_{e_1pe_1}(e_1) =0,&p\in \Omega,\\
&\gamma_{e_j}(e_i^*) + \beta_{e_i}(e_j) = 0,\\
&\beta_p(e_i^*) + (1-\delta_{1,j})\rho_{pe_je_j^*}(e_i^*) + \beta_{e_ipe_j}(e_j) = 0,& p\in \Omega,\\
&\rho_{pe_j^*}(e_i^*) + \beta_{e_ip}(e_j) = 0, & p\in\Omega,\,p_z \ne e_j, \\
&\alpha_v(e_i^*)+(1-\delta_{1,j})\rho_{e_je_j^*}(e_i^*) + \beta_{e_ie_j}(e_j) = 0,\\
&\gamma_{e_jpe_i}(e_i^*) + \gamma_p(e_j) + (1-\delta_{1,i})\rho_{e_i(pe_i)^*}(e_j) = 0, & p\in \Omega,\\
& \gamma_{e_jp}(e_i^*)+\rho_{e_ip^*}(e_j)=0,& p\in\Omega,\,p_z \ne e_i,\\
&\alpha_v(e_j)+\gamma_{e_je_i}(e_i^*) + (1-\delta_{1,i})\rho_{e_ie_i^*}(e_j) = 0,\\
&\rho_{w(e_jh)^*}(e_i^*) + \rho_{e_iwh^*}(e_j)=0,& wh^* \in\mathfrak{M},
\end{align*}
for any $1\le i, j \le \ell$.
\end{theorem}
\begin{proof} We have to proof only first and second equations, because another equations have been obtained in Theorem \ref{genth}.

\smallskip

\par 1) For the special edge $e_1$ we have:
\begin{eqnarray*}
A(e_1) &=& \sum\limits_{p\in\Omega}\nu_p\Bigl(pe_1-e_1p\Bigr) + \sum\limits_{p\in\Omega}\nu'_p\Bigl(\delta_{p_0,e_1}(p/p_0)^* - e_1p^*\Bigr)+\\&&+
\sum\limits_{wh^*\in\mathfrak{M}}\nu''_{wh^*}\Bigl(\delta_{h_0,e_1}w(h/h_0)^* - e_1wh^*\Bigr)=\\
&=& \sum\limits_{p\in\Omega}\nu_p\Bigl(pe_1-e_1p\Bigr) + \nu'_{e_1}\Bigl(v-e_1e_1^*\Bigr) - \sum\limits_{k=2}^\ell\nu'_{e_k}e_1e_k^* + \nu_{e_1e_1}\Bigl(e_1^* - e_1e_1^*e_1^*\Bigr)-\\&&-
\sum\limits_{k=2}^\ell\nu'_{e_ke_1}e_1e_1^*e_k^* + \sum\limits_{k=2}^\ell\nu'_{e_1e_k}\Bigl(e_k^* - e_1(e_1e_k)^*\Bigr) -\sum\limits_{k=2}^\ell\sum\limits_{r=2}^\ell\nu'_{e_ke_r}e_1(e_ke_r)^*+\\&&+
\sum\limits_{p\in\Omega}\nu'_{e_1pe_1}\Bigl((pe_1)^* - e_1e_1^*(e_1p)^*\Bigr)-\sum\limits_{k=2}^\ell\sum\limits_{p\in\Omega}\nu'_{e_kpe_1}e_1e_1^*(e_kp)^*+\\&&+
\sum\limits_{k=2}^\ell\sum\limits_{p\in\Omega}\nu'_{e_1pe_k}\Bigl((pe_k)^* - e_1(e_1pe_k)^*\Bigr)-
\sum\limits_{k=2}^\ell\sum\limits_{r=2}^\ell\sum\limits_{p\in\Omega}\nu'_{e_kpe_r}e_1(e_kpe_r)^*+\\&&+
\sum\limits_{we_1^*\in\mathfrak{M}}\nu''_{we_1^*}\Bigl(w - e_1we_1^*\Bigr) - \sum\limits_{k=2}^\ell\sum\limits_{w\in\Omega}\nu''_{we_k^*}e_1we_k^* +\\
&&+\sum\limits_{wh^*\in\mathfrak{M}}\nu''_{w(e_1h)^*}\Bigl(wh^* - e_1w(e_1h)^*\Bigr)
-\sum\limits_{k=2}^\ell\sum\limits_{wh^*\in\mathfrak{M}}\nu''_{w(e_kh)^*}e_1w(e_kh)^*.
\end{eqnarray*}

\smallskip

\par Using the equation $e_1e_1^* = v - \sum\limits_{k=2}^\ell e_ke_k^*$ we get:
\begin{eqnarray*}
A(e_1) &=&\sum\limits_{p\in\Omega}\nu_p\Bigl(pe_1-e_1p\Bigr) +\sum\limits_{k=2}^\ell\nu'_{e_1}e_ke_k^*- \sum\limits_{k=2}^\ell\nu'_{e_k}e_1e_k^* + \sum\limits_{k=2}^\ell\nu_{e_1e_1}e_ke_k^*e_1^*-\\&&-
\sum\limits_{k=2}^\ell\nu'_{e_ke_1}e_k^*+\sum\limits_{r=2}^\ell\sum\limits_{k=2}^\ell\nu'_{e_ke_1}e_r(e_ke_r)^* + \sum\limits_{k=2}^\ell\nu'_{e_1e_k}\Bigl(e_k^* - e_1(e_1e_k)^*\Bigr) -\sum\limits_{k=2}^\ell\sum\limits_{r=2}^\ell\nu'_{e_ke_r}e_1(e_ke_r)^*+\\&&+
\sum\limits_{p\in\Omega}\nu'_{e_1pe_1}\Bigl((pe_1)^* - (e_1p)^*\Bigr)+ \sum\limits_{k=2}^\ell\sum\limits_{p\in\Omega}\nu'_{e_1pe_1}e_k(e_1pe_k)^*  -\sum\limits_{k=2}^\ell\sum\limits_{p\in\Omega}\nu'_{e_kpe_1}(e_kp)^*+\\&&+
\sum\limits_{k=2}^\ell\sum\limits_{r=2}^\ell\sum\limits_{p\in\Omega}\nu'_{e_kpe_1}e_r(e_kpe_r)^*+
\sum\limits_{k=2}^\ell\sum\limits_{p\in\Omega}\nu'_{e_1pe_k}\Bigl((pe_k)^* - e_1(e_1pe_k)^*\Bigr)-\\&&-
\sum\limits_{k=2}^\ell\sum\limits_{r=2}^\ell\sum\limits_{p\in\Omega}\nu'_{e_kpe_r}e_1(e_kpe_r)^*+
\sum\limits_{we_1^*\in\mathfrak{M}}\nu''_{we_1^*}\Bigl(w - e_1we_1^*\Bigr) - \sum\limits_{k=2}^\ell\sum\limits_{w\in\Omega}\nu''_{we_k^*}e_1we_k^* +\\
&&+\sum\limits_{wh^*\in\mathfrak{M}}\nu''_{w(e_1h)^*}\Bigl(wh^* - e_1w(e_1h)^*\Bigr)
-\sum\limits_{k=2}^\ell\sum\limits_{wh^*\in\mathfrak{M}}\nu''_{w(e_kh)^*}e_1w(e_kh)^*.
\end{eqnarray*}

\smallskip

\par We have:

\smallskip

\begin{eqnarray*}
\Bigl.A(e_1)\Bigr|_{\{v\}} & = & 0,\\
\Bigl.A(e_1)\Bigr|_\Omega &=& \sum\limits_{p\in\Omega}\nu_{p}\Bigl(pe_1-e_1p\Bigr) + \sum\limits_{we_1^*\in\mathfrak{M}}\nu''_{we_1^*}w,
\end{eqnarray*}
it follows that there are no terms of form $e_1pe_1$, i.e.,
\[
\beta_{e_1pe_1}(e_1) = 0.
\]
\smallskip

\par Further, we have:
\begin{eqnarray*}
\Bigl.A(e_1)\Bigr|_{\Omega^*} &=& \sum\limits_{p\in\Omega}\nu'_{e_1pe_1}\Bigl((pe_1)^* - (e_1p)^*\Bigr) + \sum\limits_{k=2}^\ell\nu'_{e_1e_k}e_k^* - \sum\limits_{k=2}^\ell\nu'_{e_ke_1}e_k^*+\\
&&+\sum\limits_{k=2}^\ell\sum\limits_{p\in\Omega}\nu'_{e_1pe_k}(pe_k)^* - \sum\limits_{k=2}^\ell\sum\limits_{p\in\Omega}\nu'_{e_kpe_1}(e_kp)^* =\\&=&\sum\limits_{k=2}^\ell\Bigl(\nu'_{e_1e_k}- \nu'_{e_ke_1}\Bigr)e_k^* +\sum\limits_{k=2}^\ell\Bigl(\nu'_{e_1e_1e_k} - \nu'_{e_1e_ke_1}\Bigr)(e_1e_k)^* +\sum\limits_{k=2}^\ell\Bigl(\nu'_{e_1e_ke_1}-\nu'_{e_ke_1e_1}\Bigr)(e_ke_1)^* + \\&&
+\sum\limits_{r=2}^\ell\sum\limits_{k=2}^\ell\Bigl(\nu'_{e_1e_re_k} - \nu'_{e_re_ke_1}\Bigr)(e_re_k)^* + \sum\limits_{p\in\Omega}\Bigl(\nu'_{e_1e_1pe_1} - \nu'_{e_1pe_1e_1}\Bigr)(e_1pe_1)^*+
\\&&
+\sum\limits_{k=2}^\ell\sum\limits_{p\in\Omega}\Bigl(\nu'_{e_1e_1pe_k}-\nu'_{e_1pe_ke_1}\Bigr)(e_1pe_k)^*+
\sum\limits_{k=2}^\ell\sum\limits_{p\in\Omega}\Bigl(\nu'_{e_1e_kpe_1}-\nu'_{e_kpe_1e_1}\Bigr)(e_kpe_1)^*+
\\&& + \sum\limits_{r=2}^\ell\sum\limits_{k=2}^\ell\sum\limits_{p\in\Omega}\Bigl(\nu'_{e_1e_rpe_k} - \nu'_{e_rpe_ke_1}\Bigr)(e_rpe_k)^*=\\&=&
\sum\limits_{k=2}^\ell\Bigl(\nu'_{e_1e_k}- \nu'_{e_ke_1}\Bigr)e_k^* +\sum\limits_{k=1}^\ell\sum\limits_{r=1}^\ell\Bigl(\nu'_{e_1e_re_k} - \nu'_{e_re_ke_1}\Bigr)(e_re_k)^*
+\\&&+
\sum\limits_{k=1}^\ell\sum\limits_{r=1}^\ell\sum\limits_{p\in\Omega}\Bigl(\nu'_{e_1e_rpe_k} - \nu'_{e_rpe_ke_1}\Bigr)(e_rpe_k)^*,
\end{eqnarray*}
it follows that there are no terms of form $e_1^*p^*e_1^*$, i.e.,
\[
\gamma_{e_1pe_1}(e_1)  = 0.
\]
\smallskip

\par Further, we have:
\begin{eqnarray*}
\Bigl.A(e_1)\Bigr|_\mathfrak{M} &=& \sum\limits_{k=2}^\ell\nu'_{e_1}e_ke_k^* -\sum\limits_{k=2}^\ell\nu'_{e_k}e_1e_k^* + \sum\limits_{r=2}^\ell\nu'_{e_1e_1}e_r(e_1e_r)^*+\\&&+
\sum\limits_{r=2}^\ell\sum\limits_{k=2}^\ell\nu'_{e_ke_1}e_r(e_ke_r)^* - \sum\limits_{r=2}^\ell\nu'_{e_1e_r}e_1(e_1e_r)^* -\\&&-
\sum\limits_{r=2}^\ell\sum\limits_{k=2}^\ell\nu'_{e_ke_r}e_1(e_ke_r)^* +\sum\limits_{r=2}^\ell\sum\limits_{p\in\Omega}\nu'_{e_1pe_1}e_r(e_1pe_r)^* +\\
&&+\sum\limits_{r=2}^\ell\sum\limits_{k=2}^\ell\sum\limits_{p\in\Omega}\nu'_{e_kpe_1}e_r(e_kpe_r)^* -\sum\limits_{r=2}^\ell\sum\limits_{p\in\Omega}\nu'_{e_1pe_r}e_1(e_1pe_r)^*-\\&&-
\sum\limits_{r=2}^\ell\sum\limits_{k=2}^\ell\sum\limits_{p\in\Omega}\nu'_{e_kpe_r}e_1(e_kpe_r)^* - \sum\limits_{we_1^*\in\mathfrak{M}}\nu''_{we_1^*}e_1we_1^*-\\&&
-\sum\limits_{k=2}^\ell\sum\limits_{w\in\Omega}\nu''_{we_k^*}e_1we_k^*+
\sum\limits_{wh^*\in\mathfrak{M}}\nu''_{w(e_1h)^*}\Bigl(wh^*-e_1w(e_1h)^*\Bigr)-\\&&-
\sum\limits_{k=2}^\ell\sum\limits_{wh^*\in\mathfrak{M}}\nu''_{w(e_kh)^*}e_1w(e_kh)^*.
\end{eqnarray*}

\smallskip

\par Let us add up similar terms:
\begin{eqnarray*}
\Bigl.A(e_1)\Bigr|_\mathfrak{M} &=& \sum\limits_{k=2}^\ell\Bigl(\nu'_{e_1} + \nu''_{e_k(e_1e_k)^*}\Bigr)e_ke_k^* + \sum\limits_{k=2}^\ell\Bigl(\nu''_{e_1(e_1e_k)^*} - \nu'_{e_k}\Bigr)e_1e_k^*+\\&&+
\sum\limits_{k=1}^\ell\sum\limits_{r=2}^\ell\Bigl(\nu'_{e_ke_1}+\nu''_{e_r(e_1e_ke_r)^*}\Bigr)e_r(e_ke_r)^*+
\sum\limits_{k=1}^\ell\sum\limits_{r=2}^\ell\Bigl(\nu''_{e_1(e_1e_ke_r)^*}-\nu'_{e_ke_r}\Bigr)e_1(e_ke_r)^*+\\&&+
\sum\limits_{k=1}^\ell\sum\limits_{r=2}^\ell\sum\limits_{p\in\Omega}\Bigl(\nu'_{e_kpe_1}+
\nu''_{e_r(e_1e_kpe_r)^*}\Bigr)e_r(e_kpe_r)^*+
\sum\limits_{k=1}^\ell\sum\limits_{r=2}^\ell\sum\limits_{p\in\Omega}\Bigl(\nu''_{e_1(e_1e_kpe_r)^*}-
\nu'_{e_kpe_r}\Bigr)e_1(e_kpe_r)^*+\\&&+
\sum\limits_{k=1}^\ell\sum\limits_{we_k^*\in\mathfrak{M}}\Bigl(\nu''_{e_1w(e_1e_k)^*} - \nu''_{we_k^*}\Bigr)e_1we_k^* +
\sum\limits_{k=1}^\ell\sum\limits_{wh^*\in\mathfrak{M}}\Bigl(\nu''_{e_1w(e_1e_kh)^*} - \nu''_{w(e_kh)^*}\Bigr)e_1w(e_kh)^*+\\&&+\sum\limits_{wh^*\in\mathfrak{M},w_0 \ne e_1,\,w \notin E}\nu''_{w(e_1h)^*}wh^*,
\end{eqnarray*}
we see there are no zero terms.

\smallskip

\par 2) Let us consider the edges $e_r\in E$, $2 \le r \le \ell$. For fixed $e_r$ we have:
\begin{eqnarray*}
A(e_r) & = & \sum\limits_{p\in\Omega}\nu_p\Bigl( pe_r- e_rp\Bigr) + \sum\limits_{p\in\Omega}\nu'_{p}\Bigl(\delta_{p_0,e_r}(p/p_0)^* - e_rp^*\Bigr) + \\&&+ \sum\limits_{wh^*\in\mathfrak{M}}\nu''_{wh^*}\Bigl(\delta_{h_0,e_r}w(h/h_0)^* - e_rwh^*\Bigr)=\\&=&\sum\limits_{p\in\Omega}\nu_p\Bigl(pe_r - e_rp\Bigr) + \nu'_{e_r}\Bigl(v -e_re_r^*\Bigr) - \sum\limits_{k=1,k\ne r}^\ell\nu'_{e_k}e_re_k^*+\\
&&+\sum\limits_{p\in\Omega}\nu'_{e_rp}\Bigl(p^* - e_r(e_rp)^*\Bigr)-\sum\limits_{k=1,k\ne r}^\ell\sum\limits_{p\in\Omega}\nu'_{e_kp}e_r(e_kp)^*+\\&&+\sum\limits_{w\in\Omega}\nu''_{we_r^*}\Bigl(w-e_rwe_r^*\Bigr) + \sum\limits_{wh^*\in\mathfrak{M}}\nu''_{w(e_rh)^*}\Bigl(wh^* - e_rw(e_rh)^*\Bigr)-\\&&-
\sum\limits_{k=1,k\ne r}^\ell\sum\limits_{we_k^*\in\mathfrak{M}}\nu''_{we_k^*}e_rwe_k^* -\sum\limits_{k=1,k\ne r}^\ell\sum\limits_{wh^*\in\mathfrak{M}}\nu''_{w(e_kw)^*}e_rw(e_kh)^*.
\end{eqnarray*}

\smallskip

\par Let us add up similar terms:
\begin{eqnarray*}
\Bigl.A(e_r)\Bigr|_{\{v\}} &=& \nu'_{e_r},\\
\Bigl.A(e_r)\Bigr|_\Omega &=& \sum\limits_{p\in\Omega} \nu_p\Bigl(pe_r - e_rp\Bigr) + \sum\limits_{p\in\Omega}\nu''_{pe_r^*}p,\\
\Bigl.A(e_r)\Bigr|_{\Omega^*} &=& \sum\limits_{p\in\Omega}\nu'_{e_rp}p^*,\\
\Bigl.A(e_r)\Bigr|_\mathfrak{M} &=& -\sum\limits_{k=1}^\ell\nu'_{e_k}e_re_k^* -\sum\limits_{k=1}^\ell\sum\limits_{p\in\Omega}\nu'_{e_kp}e_r(e_kp)^*-
\sum\limits_{k=1}^\ell\sum\limits_{we_k^*\in\mathfrak{M}}\nu''_{we_k^*}e_rwe_k^*-\\&&-
\sum\limits_{k=1}^\ell\sum\limits_{wh^*\in\mathfrak{M}}\nu''_{w(e_kh)^*}e_rw(e_kh)^* + \sum\limits_{wh^*\in\mathfrak{M}}\nu''_{w(e_rh)^*}wh^*=\\&=&
\sum\limits_{k=1}^\ell\sum\limits_{wh^*\in\mathfrak{M}}\Bigl(\nu''_{e_r(e_re_k)^*} - \nu'_{e_k}\Bigr)e_re_k^* + \sum\limits_{k=1}^\ell\sum\limits_{p\in\Omega}\Bigl(\nu''_{e_r(e_re_kp)^*} - \nu'_{e_kp}\Bigr)e_r(e_kp)^*+\\&&+
\sum\limits_{k=1}^\ell\sum\limits_{we_k^*\in\mathfrak{M}}\Bigl(\nu''_{e_rw(e_re_k)^*} - \nu''_{we_k^*}\Bigr)e_rwe_k^* + \sum\limits_{k=1}^\ell\sum\limits_{wh^*\in\mathfrak{M}}\Bigl(\nu''_{e_rw(e_re_kh)^*} - \nu''_{w(e_kh)^*}\Bigr)e_rw(e_kh)^*+\\&&+\sum\limits_{wh^*\in\mathfrak{M},w_0 \ne e_k}\nu''_{w(e_rh)^*}wh^*,
\end{eqnarray*}
we see there are no zero terms.

\smallskip

\par 4) For the $e_1^*$ we have:
\begin{eqnarray*}
A(e_1^*) &=& \sum\limits_{p\in\Omega}\nu_p\Bigl(pe_1^* - \delta_{p_0,e_1}(p/p_0)\Bigr) + \sum\limits_{p\in\Omega^*}\nu'_p\Bigl((e_1p)^* - (pe_1)^*\Bigr) + \\&&+ \sum\limits_{wh^*\in\mathfrak{M}}\nu''_{wh^*}\Bigl(w(e_1h)^* - \delta_{w_0,e_1}(w/w_0)h^*\Bigr) =\\&=& \nu_{e_1}\Bigl(e_1e_1^* - v\Bigr) + \sum\limits_{k=2}^\ell\nu_{e_k}e_ke_1^* + \nu_{e_1e_1}\Bigl(e_1e_1e_1^* - e_1\Bigr)+\\&&+ \sum\limits_{k=2}^\ell\nu_{e_ke_1}e_ke_1e_1^* +\sum\limits_{k=2}^\ell\nu_{e_1e_k}\Bigl(e_1e_ke_1^* - e_k\Bigr) + \sum\limits_{k=2}^\ell\sum\limits_{r=2}^\ell\nu_{e_ke_r}e_ke_re_1^*+ \\
&&+\sum\limits_{p\in\Omega}\nu_{e_1pe_1}\Bigl(e_1pe_1e_1^* - pe_1\Bigr) + \sum\limits_{k=2}^\ell\sum\limits_{p\in\Omega}\nu_{e_kpe_1}e_kpe_1e_1^* + \\&&+ \sum\limits_{k=2}^\ell\sum\limits_{p\in\Omega}\nu_{e_1pe_k}\Bigl(e_1pe_ke_1^* - pe_k\Bigr) +
\sum\limits_{k=2}^\ell\sum\limits_{r=2}^\ell\sum\limits_{p\in\Omega}\nu_{e_kpe_r}e_kpe_re_1^*+ \\&&+\sum\limits_{p\in\Omega}\nu'_p\Bigl((e_1p)^* - (pe_1)^*\Bigr) + \sum\limits_{e_1h^*\in\mathfrak{M}}\nu''_{e_1h^*}\Bigl(e_1(e_1h)^* - h^*\Bigr) + \sum\limits_{k=2}^\ell\sum\limits_{h\in\Omega}\nu''_{e_kh^*}e_k(e_1h)^*
+\\&&+\sum\limits_{wh^*\in\mathfrak{M}}\nu''_{e_1wh^*}\Bigl(e_1w(e_1h)^* - wh^*\Bigr)+
\sum\limits_{k=2}^\ell\sum\limits_{wh^*\in\mathfrak{M}}\nu''_{e_kwh^*}e_kw(e_1h)^*.
\end{eqnarray*}

\smallskip

\par Using the equation $e_1e_1^* = v -\sum\limits_{k=2}^\ell e_ke_k^*$, we get:
\begin{eqnarray*}
A(e_1^*) &=& -\sum\limits_{k=2}^\ell\nu_{e_1}e_ke_k^* + \sum\limits_{k=2}^\ell\nu_{e_k}e_ke_1^* - \sum\limits_{k=2}^\ell \nu_{e_1e_1}e_1e_ke_k^* +\sum\limits_{k=2}^\ell\nu_{e_ke_1}e_k -\sum\limits_{k=2}^\ell\sum\limits_{r=2}^\ell\nu_{e_ke_1}e_ke_re_r^*+\\&&+ \sum\limits_{k=2}^\ell\nu_{e_1e_k}e_1e_ke_1^* - \sum\limits_{k=2}^\ell\nu_{e_1e_k}e_k + \sum\limits_{k=2}^\ell\sum\limits_{r=2}^\ell\nu_{e_ke_r}e_ke_re_1^*
+\sum\limits_{p\in\Omega}\nu_{e_1pe_1}\Bigl(e_1p - pe_1\Bigr) - \\&&-
\sum\limits_{k=2}^\ell\sum\limits_{p\in\Omega}\nu_{e_1pe_1}e_1pe_ke_k^* + \sum\limits_{k=2}^\ell\sum\limits_{p\in\Omega}\nu_{e_kpe_1}e_kp -
\sum\limits_{k=2}^\ell\sum\limits_{r=2}^\ell\sum\limits_{p\in\Omega}\nu_{e_kpe_1}e_kpe_re_r^*+\\&&+
\sum\limits_{k=2}^\ell\sum\limits_{p\in\Omega}\nu_{e_1pe_k}\Bigl(e_1pe_ke_1^* - pe_k\Bigr) +
\sum\limits_{k=2}^\ell\sum\limits_{r=2}^\ell\sum\limits_{p\in\Omega}\nu_{e_kpe_r}e_kpe_re_1^*+ \\&&+\sum\limits_{p\in\Omega}\nu'_p\Bigl((e_1p)^* - (pe_1)^*\Bigr) + \sum\limits_{e_1h^*\in\mathfrak{M}}\nu''_{e_1h^*}\Bigl(e_1(e_1h)^* - h^*\Bigr) +\sum\limits_{k=2}^\ell\sum\limits_{h\in\Omega}\nu''_{e_kh^*}e_k(e_1h)^*+\\
&&+\sum\limits_{wh^*\in\mathfrak{M}}\nu''_{e_1wh^*}\Bigl(e_1w(e_1h)^* - wh^*\Bigr)+
\sum\limits_{k=2}^\ell\sum\limits_{wh^*\in\mathfrak{M}}\nu''_{e_kwh^*}e_kw(e_1h)^*.
\end{eqnarray*}

\smallskip

\par Let us add up similar terms:

\begin{eqnarray*}
\Bigl.A(e_1^*)\Bigr|_\Omega & = & \sum\limits_{k=2}^\ell \nu_{e_ke_1}e_k-\sum\limits_{k=2}^\ell\nu_{e_1e_k}e_k + \sum\limits_{p\in\Omega}\nu_{e_1pe_1}\Bigl(e_1p-pe_1\Bigr) + \sum\limits_{k=2}^\ell\sum\limits_{p\in\Omega}\nu_{e_kpe_1}e_kp -\\&&- \sum\limits_{k=2}^\ell\sum\limits_{p\in\Omega}\nu_{e_1pe_k}pe_k = \\
&=& \sum\limits_{k=2}^\ell\Bigl(\nu_{e_ke_1} -\nu_{e_1e_k}\Bigr)e_k + \sum\limits_{k=2}^\ell\Bigl(\nu_{e_1e_ke_1}-\nu_{e_1e_1e_k}\Bigr)e_1e_k + \sum\limits_{k=2}^\ell\Bigl(\nu_{e_ke_1e_1} - \nu_{e_1e_ke_1}\Bigr)e_ke_1+\\&&
+\sum\limits_{k=2}^\ell\sum\limits_{r=2}^\ell\Bigl(\nu_{e_ke_re_1} - \nu_{e_1e_ke_r}\Bigr)e_ke_r+
\sum\limits_{k=2}^\ell\sum\limits_{p\in\Omega}\Bigl(\nu_{e_1pe_ke_1} - \nu_{e_1e_1pe_k}\Bigr)e_1pe_k +\\
&&+ \sum\limits_{k=2}^\ell\sum\limits_{p\in\Omega}\Bigl(\nu_{e_kpe_1e_1} -\nu_{e_1e_kpe_1}\Bigr)e_kpe_1 + \sum\limits_{k=2}^\ell\sum\limits_{r=2}^\ell\sum\limits_{p\in\Omega}\Bigl(\nu_{e_kpe_re_1} - \nu_{e_1e_kpe_r}\Bigr)e_kpe_r=\\
&=& \sum\limits_{k=2}^\ell\Bigl(\nu_{e_ke_1} -\nu_{e_1e_k}\Bigr)e_k + \sum\limits_{k=1}^\ell\sum\limits_{r=1}^\ell\Bigl(\nu_{e_ke_re_1} - \nu_{e_1e_ke_r}\Bigr)e_ke_r+\\&&+\sum\limits_{k=1}^\ell\sum\limits_{r=1}^\ell\sum\limits_{p\in\Omega}\Bigl(\nu_{e_kpe_re_1} - \nu_{e_1e_kpe_r}\Bigr)e_kpe_r.
\end{eqnarray*}

\smallskip

\par It follows that there are not terms of form $e_1pe_1$, i.e.,
\[
\beta_{e_1pe_1}(e_1^*) = 0.
\]

\smallskip

\par Further, we have:

\begin{eqnarray*}
\Bigl.A(e_1^*)\Bigr|_{\Omega^*} & = & \sum\limits_{p\in\Omega}\nu'_p\Bigl((e_1p)^*-(pe_1)^*\Bigr) - \sum\limits_{e_1h^*\in\mathfrak{M}}\nu''_{e_1h^*}h^*=\\&=&
-\sum\limits_{k=2}^\ell\nu''_{e_1e_k^*}e_k^* -\sum\limits_{k=2}^\ell\nu'_{e_k}(e_ke_1)^* +
\sum\limits_{k=2}^\ell\Bigl(\nu'_{e_k} - \nu''_{e_1(e_1e_k)^*}\Bigr)(e_1e_k)^*-\sum\limits_{k=2}^\ell\sum\limits_{r=2}^\ell\nu''_{e_1(e_ke_r)^*}(e_ke_r)^*+\\
&&+\sum\limits_{p\in\Omega}\Bigl(\nu'_{pe_1} - \nu'_{e_1p}\Bigr)(e_1pe_1)^* -\sum\limits_{k=2}^\ell\sum\limits_{p\in\Omega}\nu'_{e_kp}(e_kpe_1)^* + \Bigl(\nu'_{pe_k} - \nu''_{e_1(e_1pe_k)^*}\Bigr)(e_1pe_k)^* -\\&&-\sum\limits_{k=2}^\ell\sum\limits_{r=2}^\ell\sum\limits_{p\in\Omega}\nu''_{e_1(e_rpe_k)^*}(e_rpe_k)^*,
\end{eqnarray*}
it follows that there are no terms of form $e_1^*p^*e_1^*$, i.e.,
\[
\gamma_{e_1pe_1}(e_1^*)  = 0.
\]

\smallskip

\par Further, we have:
\begin{eqnarray*}
\Bigl.A(e_1^*)\Bigr|_\mathfrak{M} & = & -\sum\limits_{k=2}^\ell\nu_{e_1}e_ke_k^* + \sum\limits_{k=2}^\ell\nu_{e_k}e_ke_1^* - \sum\limits_{k=2}^\ell \nu_{e_1e_1}e_1e_ke_k^* -\sum\limits_{k=2}^\ell\sum\limits_{r=2}^\ell\nu_{e_ke_1}e_ke_re_r^*+\\&&+ \sum\limits_{k=2}^\ell\nu_{e_1e_k}e_1e_ke_1^* +\sum\limits_{k=2}^\ell\sum\limits_{r=2}^\ell\nu_{e_ke_r}e_ke_re_1^*-
\sum\limits_{k=2}^\ell\sum\limits_{p\in\Omega}\nu_{e_1pe_1}e_1pe_ke_k^* -\\&&
-\sum\limits_{k=2}^\ell\sum\limits_{r=2}^\ell\sum\limits_{p\in\Omega}\nu_{e_kpe_1}e_kpe_re_r^*+
\sum\limits_{k=2}^\ell\sum\limits_{p\in\Omega}\nu_{e_1pe_k}e_1pe_ke_1^* +
\sum\limits_{k=2}^\ell\sum\limits_{r=2}^\ell\sum\limits_{p\in\Omega}\nu_{e_kpe_r}e_kpe_re_1^*+ \\ &&+\sum\limits_{e_1h^*\in\mathfrak{M}}\nu''_{e_1h^*}e_1(e_1h)^*+
\sum\limits_{k=2}^\ell\sum\limits_{h\in\Omega}\nu''_{e_kh^*}e_k(e_1h)^* +\\&&+
\sum\limits_{wh^*\in\mathfrak{M}}\nu''_{e_1wh^*}\Bigl(e_1w(e_1h)^* - wh^*\Bigr)+
\sum\limits_{k=2}^\ell\sum\limits_{wh^*\in\mathfrak{M}}\nu''_{e_kwh^*}e_kw(e_1h)^*.
\end{eqnarray*}

\smallskip

\par Let us add up similar terms:
\begin{eqnarray*}
\Bigl.A(e_1^*)\Bigr|_\mathfrak{M} & = & \sum\limits_{k=2}^\ell\Bigl(-\nu_{e_1} - \nu''_{e_1e_ke_k^*}\Bigr)e_ke_k^* +\sum\limits_{k=2}^\ell\Bigl(\nu_{e_k} -\nu''_{e_1e_ke_1^*}\Bigr)e_ke_1^*+\\&&+
\sum\limits_{k=1}^\ell\sum\limits_{r=2}^\ell\Bigl(-\nu_{e_ke_1} - \nu''_{e_1e_ke_re_r^*}\Bigr)e_ke_re_r^* +\sum\limits_{k=1}^\ell\sum\limits_{r=2}^\ell\Bigl(\nu_{e_ke_r}-\nu''_{e_1e_ke_re_1^*}\Bigr)e_ke_re_1^* +\\&&
+\sum\limits_{k=1}^\ell\sum\limits_{r=2}^\ell\sum\limits_{p\in\Omega}\Bigl(-\nu_{e_kpe_1}-
\nu''_{e_1e_kpe_re_r^*}\Bigr)e_kpe_re_r^*
+\sum\limits_{k=1}^\ell\sum\limits_{r=2}^\ell\sum\limits_{p\in\Omega}\Bigl(\nu_{e_kpe_r} - \nu''_{e_1e_kpe_re_1^*}\Bigr)e_kpe_re_1^* + \\&& +
\sum\limits_{k=1}^\ell\sum\limits_{e_1h^*\in\mathfrak{M}}\Bigl(\nu''_{e_kh^*}-\nu''_{e_1e_k(e_1h)^*}\Bigr)e_k(e_1h)^*+
\sum\limits_{k=1}^\ell\sum\limits_{wh^*\in\mathfrak{M}}\Bigl(\nu''_{e_kwh^*}-\nu''_{e_1e_kw(e_1h)^*}\Bigr)e_kw(e_1h)^*-\\&&-
\sum\limits_{wh^*\in\mathfrak{M},h_0 \ne e_1}\nu''_{e_1wh^*}wh^*,
\end{eqnarray*}
we see there are no zero terms.

\smallskip

\par 4) For $e_r^*\in E^*$, $2 \le r \le \ell$, we have:

\begin{eqnarray*}
A(e_r^*) &=& \sum\limits_{p\in\Omega}\nu_p\Bigl(pe_r^* - \delta_{p_0,e_r}(p/p_0)\Bigr)+ \sum\limits_{p\in\Omega}\nu'_p\Bigl((e_rp)^* - (pe_r)^*\Bigr)+\\&&+
\sum\limits_{wh^*\in\mathfrak{M}}\nu''_{wh^*}\Bigl(w(e_rh)^* - \delta_{w_0,e_r}(w/w_0)h^*\Bigr)=\\&=&
\nu_{e_r}\Bigl(e_re_r^* - v\Bigr) + \sum\limits_{k=1,k\ne r}^\ell\nu_{e_k}e_ke_r^* + \\&& +\sum\limits_{p\in\Omega}\nu_{e_rp}\Bigl(e_rpe_r^* - p\Bigr) +\sum\limits_{k=1,k\ne r}^\ell\sum\limits_{p\in\Omega}\nu_{e_kp}e_kpe_r^*+\\&&+
\sum\limits_{p\in\Omega}\nu'_p\Bigl((e_rp)^* - (pe_r)^*\Bigr) + \sum\limits_{h\in\Omega}\nu''_{e_rh^*}\Bigl(e_r(e_rh)^* - h^*\Bigr) + \\&&+
\sum\limits_{k=1,k\ne r}^\ell\nu''_{e_kh^*}e_k(e_rh)^*+\sum\limits_{wh^*\in\mathfrak{M}}\nu''_{e_rwh^*}\Bigl(e_rw(e_rh)^* - wh^*\Bigr)+\\&&+\sum\limits_{k=1,k\ne r}^\ell\sum\limits_{wh^*\in\mathfrak{M}}\nu''_{e_kwh^*}e_kw(e_rh)^*.
\end{eqnarray*}

\smallskip

\par Let us add up similar terms:
\begin{eqnarray*}
\left.A(e_r^*)\right|_{\{v\}} &=& - \nu_{e_r},\\
\left.A(e_r^*)\right|_\Omega &=& - \sum\limits_{p\in\Omega}\nu_{e_rp}p,\\
\left.A(e_r^*)\right|_{\Omega^*} &=& \sum\limits_{p\in\Omega}\nu_{p}'\Bigl((e_rp)^* - (pe_r)^*\Bigr) - \sum\limits_{h\in\Omega}\nu''_{e_rh^*}h^*=\\&=&
-\sum\limits_{k=1}^\ell\nu''_{e_re_k^*}e_k^* + \sum\limits_{k=1,k\ne r}^\ell\Bigl(-\nu'_{e_k} - \nu''_{e_r(e_ke_r)^*}\Bigr)(e_ke_r)^* + \sum\limits_{k=1,k\ne r}^\ell\Bigl(\nu'_{e_k} - \nu''_{e_r(e_re_k)^*}\Bigr)(e_re_k)^*-\\&&-
\nu''_{e_r(e_re_r)^*}(e_re_r)^* - \sum\limits_{k=1,k\ne r}^\ell\sum\limits_{t=1,t \ne r}^\ell\nu''_{e_r(e_ke_t)^*}(e_ke_t)^* + \sum\limits_{p\in\Omega}\Bigl(\nu'_{pe_r} - \nu'_{e_rp} - \nu''_{e_r(e_rpe_r)^*}\Bigr)(e_rpe_r)^*+\\&&+
\sum\limits_{k=1,k\ne r}^\ell\sum\limits_{p\in\Omega}\Bigl(-\nu'_{e_kp} -\nu''_{e_r(e_kpe_r)^*}\Bigr)(e_kpe_r)^* + \sum\limits_{k=1,k\ne r}^\ell\sum\limits_{p\in\Omega}\Bigl(\nu'_{pe_k} - \nu''_{e_r(e_rpe_k)^*}\Bigr)(e_rpe_k)^*-\\&&-\sum\limits_{k=1,k\ne r}^\ell\sum\limits_{t=1,t\ne r}^\ell\sum\limits_{p\in\Omega}\nu''_{e_r(e_kpe_t)^*},
\end{eqnarray*}
and
\begin{eqnarray*}
\left.A(e_r^*)\right|_\mathfrak{M} &=& \sum\limits_{k=1}^\ell\nu_{e_k}e_ke_r^* + \sum\limits_{k=1}^\ell\sum\limits_{p\in\Omega}\nu_{e_kp}e_kpe_r^* + \\&&+
\sum\limits_{k=1}^\ell\sum\limits_{h\in\Omega}\nu''_{e_kh^*}e_k(e_rh)^* + \sum\limits_{k=1}^\ell\sum\limits_{wh^*\in\mathfrak{M}}\nu''_{e_kwh^*}e_kw(e_rh)^* -\\&&-
\sum\limits_{wh^*\in\mathfrak{M}}\nu''_{e_rwh^*}wh^*=\\&=&
\sum\limits_{k=1}^\ell\Bigl(\nu_{e_k} -\nu''_{e_re_ke_r^*}\Bigr)e_ke_r^* + \sum\limits_{k=1}^\ell\sum\limits_{p\in\Omega}\Bigl(\nu_{e_kp}-\nu'_{e_re_kpe_r^*}\Bigr)e_kpe_r^* +\\&&+\sum\limits_{k=1}^\ell\sum\limits_{p\in\Omega}\Bigl(\nu''_{e_kp^*} - \nu''_{e_re_k(e_rp)^*}\Bigr)e_k(e_rp)^* + \sum\limits_{k=1}^\ell\sum\limits_{wh^*\in\mathfrak{M}}\Bigl(\nu''_{e_kwh^*} - \nu''_{e_re_kw(e_rh)^*}\Bigr)e_kw(e_rh)^* -\\&&-\sum\limits_{k=1,k\ne r}^\ell\sum\limits_{wh^*\in\mathfrak{M}}\nu''_{e_rw(e_kh)^*}w(e_kh)^*,
\end{eqnarray*}
we see there are no zero terms. It completes the proof.
\end{proof}

\smallskip

\par As corollary of this Theorem follows the full description of all outer derivations of the Leavitt path algebra $W(\ell)$.

\smallskip

\begin{theorem}\label{outerder}
Any outer derivation $\mathscr{D}$ of the Leavitt path algebra $W(\ell)$ can be described as follows:
\[
\mathscr{D}(x) = \begin{cases}0, \mbox{ if $x =v$},\\
\alpha_v(x)v + \sum\limits_{p\in\Omega}\Bigl(\beta_p(x)p+\gamma_p(x)p^*\Bigr)+\sum\limits_{wh^*\in\mathfrak{M}}\rho_{wh^*}(x)wh^*,\mbox{ if $x \in E\cup E^*$,}
\end{cases}
\]
where almost all scalars $\alpha(x),\beta(x), \gamma(x), \rho(x)\in R$ are zero and they satisfy the following equations,
\begin{align*}
&\gamma_{e_j}(e_i^*) + \beta_{e_i}(e_j) = 0,\\
&\beta_p(e_i^*) + (1-\delta_{1,j})\rho_{pe_je_j^*}(e_i^*) + \beta_{e_ipe_j}(e_j) = 0,& p\in \Omega,\\
&\rho_{pe_j^*}(e_i^*) + \beta_{e_ip}(e_j) = 0, & p\in\Omega,\,p_z \ne e_j, \\
&\alpha_v(e_i^*) +(1-\delta_{1,j})\rho_{e_je_j^*}(e_i^*) + \beta_{e_ie_j}(e_j) = 0,\\
&\gamma_{e_jpe_i}(e_i^*) + \gamma_p(e_j) + (1-\delta_{1,i})\rho_{e_i(pe_i)^*}(e_j) = 0, & p\in \Omega,\\
& \gamma_{e_jp}(e_i^*)+\rho_{e_ip^*}(e_j)=0,&p\in\Omega,\,p_z \ne e_i,\\
&\alpha_v(e_j)+\gamma_{e_je_i}(e_i^*) + (1-\delta_{1,i})\rho_{e_ie_i^*}(e_j) = 0,\\
&\rho_{w(e_jh)^*}(e_i^*) + \rho_{e_iwh^*}(e_j)=0,& wh^* \in\mathfrak{M},
\end{align*}
for any $1\le i,j\le\ell$, $p\in\Omega$, and at least one of the following scalars $\beta_{e_1pe_1}(e_1)$, $\beta_{e_1pe_1}(e_1^*)$, $\gamma_{e_1pe_1}(e_1)$, $\gamma_{e_1pe_1}(e_1^*)$ are not zero.
\end{theorem}
\begin{proof}
It immediately follows from Theorem \ref{genth} and Theorem \ref{innerder}.
\end{proof}

\smallskip

\paragraph{Acknowledgements.} The author would like to express his deepest gratitude to Professor Leonid A. Bokut, who has drawn the author's attention to this work. I am also extremely indebted to my friend my Chinese Brother Zhang Junhuai for great support, without which the author's life would be very difficult.


\begin{thebibliography}{10}

\bibitem{Lev} G. Abrams and A. G. Pino, The Leavitt path algebra of a graph, {\it J. Algebra.} {\bf 293}, (2005) 319--334.

\bibitem{Lev2} P. Ara, M.A. Moreno and E. Pardo, Nonstable $K$-theory for graph algebras, {\it ALgebr. Represent Theory.} {\bf 10}(2), (2007), 157--178.

\bibitem{Zel} A. Alahmedia, H. Alsulamia, S. Jaina and E. I. Zelmanov, Leavitt Path Algebras of Finite Gelfand--Kirillov Dimension, {\it Journal of Algebra and Its Applications.} {\bf 11}(6), (2012), 1250225.

\bibitem{AraCortinas} P. Ara and G. Corti\~nas, Tensor product of Leavitt path algebras, {\it Proc. Am. Math. Soc.} {\bf 141}(8), (2013), 2629--2639.

\bibitem{AraGood} P. Ara and K. R. Goodearl, Leavitt path algebras of separated graphs,{\it J. Reine Angew. Math.} {\bf 669}, (2012), 165--224.

\bibitem{BokSurv} L.A. Bokut and Y. Chen, Gr\"obner--Shirshov basis and their calculation, {\it Bull. Math. Sci.} {\bf 4}, (2014), 325--395.
\end{thebibliography}
\end{document}